\documentclass[12pt]{amsart}
\usepackage{txfonts}
\usepackage{amsfonts}
\usepackage{amsfonts}
\usepackage{amssymb,latexsym}
\usepackage{enumerate}
\newtheorem{theorem}{Theorem}

\newtheorem{lemma}[theorem]{Lemma}
\newtheorem{definition}[theorem]{Definition}

\newcommand {\p}{\partial}

\numberwithin{equation}{section}
\numberwithin{theorem}{section}

\begin{document}
\title  [Hessian equation]
{Hessian equations on closed Hermitian manifolds}
\author{Dekai Zhang}
\address{University of Science and Technology of China\\
         Hefei, 230026, Anhui Province, CHINA}
\email{dekzhang@mail.ustc.edu.cn}
\maketitle

\begin{abstract}
In this paper, using the technical tools in \cite{TW5}, we solve the complex Hessian equation on closed Hermitian manifolds, which generalizes the the K\"ahler case results  in \cite{HMW} and \cite{DK}.
\end{abstract}

\section{Introduction}

Let $(M,g)$ be a compact Hermitian manifold of complex dimension $n\geq2$, and $\omega$ be the corresponding Hermitian form. In local coordinates, we write $\omega$ as
\begin{center}
 $\omega=\sqrt{-1}\sum\limits_{i,j = 1}^n g_{i\overline j } dz^i
 \wedge  dz^{\overline j } .$
\end{center}
In this paper, we consider the following Hessian equation on closed Hermitian manifolds
\begin{align}\label{equation}
\left\{ {\begin{array}{*{20}c}
   {C^k_n{{\omega_u^k}\wedge \omega^{n-k} }= {e^f}\omega^n}, \ \ \mathop {\sup }   \limits_M u= 0   \\
   {\omega_u=\omega+\sqrt{-1}\p\bar\p u \in \Gamma_k(M),}  \\
\end{array}} \right.
\end{align}
where $\Gamma_k(M)$ is a convex cone defined in \eqref{kcone} in section 2.

When $k=n$, the condition $\omega_u \in \Gamma_k(M)$ is equivalent to $\omega_u > 0$.
 Equation \eqref{equation} becomes  the following  Monge-Ampere equation
\begin{align}\label{M-A}
 \omega_u^n= {e^f}\omega^n,  \ \ \mathop {\sup }   \limits_M u= 0.
\end{align}
In addition, when $\left(M,\omega\right)$ is a K\"ahler manifold, i.e., $d\omega=0$, Yau \cite{Y} solved the equation \eqref{M-A} now known as Calabi-Yau theorem. For general Hermitian manifolds, the equation \eqref{equation} has been solved by Cherrier \cite{C} in the case of dimensions 2 and Tosatti-Weinkove \cite{TW2} for arbitrary dimension. For further background, we refer the reader to \cite{TW1}, \cite{TW2},  \cite{GL}, \cite{Zh} and the references therein.

   When $2\le k \le n-1$, $\omega_u$ may not be positive, the analysis becomes more complicated. Suppose that $\left(M,\omega\right)$ is a K\"ahle manifold and $\omega_u \in \Gamma_k(M)$ which is defined in section 2 , Hou-Ma-Wu \cite{HMW} proved the following second order estimates of the equation \eqref{equation}
\begin{align}\label{C2estimate}
  \max|\p\bar\p u|_g\leq C(1+\max|\nabla u|^2_g).
\end{align}
 They also pointed out in their paper that \eqref{C2estimate} may be adapted to the blowing up analysis. Later on, Dinew-Kolodziej \cite{DK} obtained the gradient estimate by \eqref{C2estimate}. Thus equation \eqref{equation} can be solved on K\"ahler manifolds under the compatible condition
\begin{center}
$\int_M{e^f \omega ^n }  = \int_M {\omega ^n }.$
\end{center}

Tosatti-Weinkove \cite{TW4} considered another Hessian typed equation related to the Gauduchon conjecture
\begin{align}\label{TWequation}
&\det\left({\omega_0}^{n-1}+{\sqrt{-1}{\partial\bar{\partial}u}}
\wedge\omega^{n-2}\right)=e^{F}\det\left(\omega^{n-1}\right)\\
&{\omega_0}^{n-1}+{\sqrt{-1}{\partial\bar{\partial}u}}
\wedge\omega^{n-2}>0,  \mathop {\sup }   \limits_M u= 0 , \notag
\end{align}
where $\omega_0$ and $\omega$ are two Hermitian metrics on $M$.

In \cite{TW4}, Tosatti-Weinkove  solved equation \eqref{TWequation} if $\omega$ is K\"ahler. One of the main parts is doing the second order estimate.  They use the similar auxiliary function in \cite{HMW}. Later on, in \cite{TW5}, they can solve \eqref{TWequation} if  $\omega$ is  Hermitian. The second order estimate becomes more difficult in the Hermitian case, the authors succeeded to obtain the second order estimates by modifying the auxiliary function in \cite{HMW}.

 In this paper, we solve  equation \eqref{equation} on closed Hermitian manifolds. More precisely, our main result is
\begin{theorem}
\ Let $(M,g)$ be a closed Hermitian manifold of complex dimension
 $n\geq2$, $f$ is a smooth real  function on $M$.
Then there is a unique real number $b$ and a unique smooth real function $u$
 on $M$ solving
\begin{align}\label{mainresult}
    &C^k_n{{\omega_u^k}\wedge \omega^{n-k} }= {e^{f+b}}\omega^n\\
 &\omega_u \in \Gamma_k(M), \ \mathop {\sup }\limits_M u = 0 .\notag
\end{align}
\end{theorem}
We use the continuity method to solve the problem \eqref{mainresult}. The openness follows from implicit function theory. The closeness argument can be reduced to \emph{a priori} estimates up to the second by the standard Evans-Krylov theory. Actually, we can derive the zero order estimate and the second order estimate of  solutions of equation
\eqref{equation} and thus use a blow up method to obtain the gradient estimate.

     In \cite{TW2}, Tosatti-Weinkove derived the key zero order estimate by proving  a Cherrier-type inequality which was originally proved in \cite{C}. For equation \eqref{equation}, we can prove the similar Cherrier-type inequality but the analysis becomes a bit complicated since $\omega_u$ may not be positive. Some inequalities for $k-$th elementary symmetric functions in \cite{CW} are needed. For the second order estimate, the main difficulty is that there are new terms of the form $T\ast D^3u$, where $T$ is the torsion of $\omega$ and $ D^3u$ represents the third derivatives of $u$. To control these terms, we use the auxiliary function due to Tosatti-Weinkove in \cite{TW5}. The main difference is that for equation \eqref{equation} we need to use some lemmas for $k-$th elementary symmetric functions proved by Hou-Ma-Wu in \cite{HMW}.

The rest of the paper is organized as follows. In section $2$, we give some preliminaries. In section $3$, the Cherrier-type inequality is derived , thus we obtain the $C^0$ estimate. In section $4$, we will prove the second order estimate by a similar auxiliary function in \cite{TW5}.

{\bf Acknowledgment:}
I would like to thank Professor Xinan Ma for his encouragement, advice and comments. I also thank Professor Shengli Kong for careful reading and many suggestions.

\section{preliminaries}
Let $(M,g)$ be a compact Hermitian manifold and let $\nabla$ denote the Chern connection of $g$. In this section we will give some  preliminaries about the $k-$th elementary symmetric function and  the commutation formula of covariant derivatives.

\subsection{Elementary symmetric function}
The $k-$th elementary symmetric function is defined by
\begin{align}
\sigma _k \left( \lambda  \right) = \sum\limits_{1 \le i_1  <  \cdots  < i_k  \le n} {\lambda _{i_1 }  \cdots \lambda _{i_k } },\notag
\end{align}
where $\lambda  = \left( {\lambda _1 , \cdots ,\lambda _n } \right)\in R^n.$
 Let $\lambda \left( {a_{i\bar j } } \right)$ denote the eigenvalues of Hermitian matrix $\left\{ {a_{i\bar j } } \right\}$, we define
\begin{align}
\sigma _k \left( {a_{i\bar j } } \right) = \sigma _k \left( {\lambda \left\{ {a_{i\bar j } } \right\}} \right).\notag
\end{align}
The definition of $\sigma _k$ can be naturally extended to Hermitian manifold. Indeed, let $A^{1,1}(M,R)$ be the space of smooth real $(1,1)$-forms on $M$, for $\chi\in A^{1,1}(M,R)$ we define
\begin{align}
\sigma _k \left( \chi  \right) = \left( {\frac{n}{k}} \right)\frac{{\chi ^k \wedge \omega ^{n - k}    }}{{\omega ^n }}.\notag
\end{align}
\begin{definition}
\begin{align}
\Gamma_k:=\{\lambda\in \mathbb{R}^{n}:\sigma_j(\lambda)>0,j=1,\cdots,k\}.
\end{align}
Similarly, we define $\Gamma_k$ on $M$ as follows
\begin{align}\label{kcone}
\Gamma_k(M):=\{\chi\in A^{1,1}\left(M,\mathbb{R}^{n}\right):\sigma_j(\chi)>0,j=1,\cdots,k\}.
\end{align}
\end{definition}
Furthermore,  $\sigma_r(\lambda|i_1\ldots i_l)$, with $i_1,\ldots, i_l$ being distinct, stands for the $r$--th symmetric function with $\lambda_{i_1} = \cdots = \lambda_{i_l} = 0$. For more details about  elementary symmetric functions, one can see the lecture notes \cite{Wang}.

\ To prove the $C^0$ estimate, we need the following lemma of elementary symmetric functions.

\begin{lemma} Suppose that $ \lambda  \in \Gamma _k , 3\le k \le n$ and $ {\lambda _1 \ge\lambda _2 \ge \cdots \ge\lambda _n } $, then there exists a positive constant $C$ depending only on $k$ and $n$, such that for $0\le i \le k-2.$
\begin{align}\label{lemma2.1}
 \left| {\lambda _{j_1 } \lambda _{j_2 }  \cdots \lambda _{j_i } } \right|\le C\sigma _i \left( {\lambda \left| j \right.} \right),\\
  1 \le j_1  < j_2  <  \cdots j_i  \le n,j_l  \ne j,1 \le l \le i,1 \le j \le n.\notag
\end{align}
\end{lemma}
\begin{proof}
Since
\[
\sum\limits_{p = k}^n {\lambda _p } {\rm{ = }}\sigma _{\rm{1}} \left( {\lambda \left| {{\rm{12}} \cdots k - 1} \right.} \right) > 0,
\]
and \[ {\lambda _1 \ge\lambda _2 \ge \cdots \ge\lambda _n } \]
then \begin{align}\label{basicinequality}
\left| {\lambda _p } \right| \le \left( {n - k} \right)\lambda _k , k + 1 \le p \le n.\end{align}
\ We first prove the lemma for $k=3$. In this case, it needs to prove that there exists a constant $C$ such that
\[|\lambda_l| \le C\sigma _1 \left( {\lambda \left| j \right.} \right),\] for $1 \le j,l \le n$ and $l  \ne j$. Indeed,
$\sigma _1 \left( {\lambda \left| j \right.} \right)=\lambda_l+\sigma _1 \left( {\lambda \left| jl \right.} \right)$, thus $\lambda_l\le\sigma _1 \left( {\lambda \left| j \right.} \right)$. Now,we assume $\lambda_l<0$, then $l\ge4$. By \eqref{basicinequality}, we have
\[\left| {\lambda _l } \right| \le \left( {n - 3} \right)\lambda _3 , 4 \le l \le n.\] Since $\lambda|j\in \Gamma_2$, by the proof in \cite{CW} which used the result in \cite{LT},  there exists a constant $\theta_1$ such that  $\sigma _1 \left( {\lambda \left| j \right.} \right)\ge \theta_1\lambda_2$ if $j=1$ and $\sigma _1 \left( {\lambda \left| j \right.} \right)\ge \theta_1\lambda_1$ if $2\le j \le n$. Taking $C=\frac{n-3}{\theta_1}$, we then prove the lemma for the case $k=3$.

Next we prove the lemma for the general $k,3\le k\le n$.

If $j > i$,  by the result in \cite{Wang}
\[\sigma _i \left( {\lambda \left| j \right.} \right) \ge \theta \left( {n,k} \right)\lambda _1  \cdots  \cdots \lambda _i.\]

Thus we have
\begin{align*}
 \left| {\lambda _{j_1 } \lambda _{j_2 }  \cdots \lambda _{j_i } } \right| =& \lambda _{j_1 }  \cdots \lambda _{j_q } \left| {\lambda _{j_q  + 1}  \cdots \lambda _{j_i } } \right| \le \lambda _1  \cdots \lambda _q \left( {n - k} \right)^{i - q} \lambda _k ^{i - q}  \\
  \le& \left( {n - k} \right)^i \lambda _1  \cdots \lambda _i  \le \frac{{\left( {n - k} \right)^i }}{{\theta \left( {n,k} \right)}}\sigma _i \left( {\lambda \left| j \right.} \right) .
 \end{align*}

If $j \le i$, then similarly
\[\sigma _i \left( {\lambda \left| j \right.} \right) \ge \theta \left( {n,k} \right)\lambda _1  \cdots \lambda _{j - 1} \lambda _{j + 1}  \cdots \lambda _{i + 1}.\]

 Thus we have

\begin{align*}
 \left| {\lambda _{j_1 } \lambda _{j_2 }  \cdots \lambda _{j_i } } \right| =& \lambda _{j_1 }  \cdots \lambda _{j_q } \left| {\lambda _{j_q  + 1}  \cdots \lambda _{j_i } } \right| \le \lambda _{1 }  \cdots \lambda _{j - 1} \lambda _{j + 1}  \cdots \lambda _{q + 1 } \left( {n - k} \right)^{i - q} \lambda _k ^{i - q}  \\
  \le& \left( {n - k} \right)^i \lambda _1  \cdots \lambda _{j - 1} \lambda _{j + 1}  \cdots \lambda _{i + 1}  \le \frac{{\left( {n - k} \right)^i }}{{\theta \left( {n,k} \right)}}\sigma _i \left( {\lambda \left| j \right.} \right) .
 \end{align*}
\end{proof}
Using this lemma, we immediately obtain the following lemma which is a key ingredient for proving lemma \ref{Ch-Ineq}.
\begin{lemma}\label{inequality}
\begin{align}\label{lemma2.2}
\sum\limits_{i = 0}^{k - 2} {\left| {\frac{{\sqrt { - 1} \partial u \wedge \bar \partial u \wedge \omega _u ^i  \wedge T_i }}{{\omega ^n }}} \right|}  \le C\sum\limits_{i = 0}^{k - 2} {\frac{{\sqrt { - 1} \partial u \wedge \bar \partial u \wedge \omega _u ^i  \wedge \omega ^{n - i - 1} }}{{\omega ^n }}},
\end{align}
,where $T_i$ is defined as the combinations of $\omega,\partial \omega, \partial\bar\partial \omega$, more precisely
\begin{align*}
T_i  = \sum\limits_{0 \le 3p + 2q \le n - i} {\omega ^{n - i - 3p - 2q}  \wedge (\sqrt { - 1})^p\left( {\partial \omega } \right)^p  \wedge \left( {\bar \partial \omega } \right)^p  \wedge (\sqrt { - 1})^q \left( {\partial \bar \partial \omega } \right)^q }
\end{align*}
\end{lemma}
\begin{proof}
For $x \in M$ ,we choose the coordinates such that
\[
\omega \left( x \right) = \sum\limits_{j = 1}^n {dz^j  \wedge d\bar z^j } ,\omega _u \left( x \right) = \sum\limits_{j = 1}^n {\lambda _j dz^j  \wedge d\bar z^j } ,\]
 and \[
 \lambda _1  \ge \lambda _2  \ge  \cdots  \ge \lambda _n.
\]

Thus we have
\begin{align}
 \sum\limits_{i = 0}^{k - 2} {\left| {\frac{{\sqrt { - 1} \partial u \wedge \bar \partial u \wedge \omega _u ^i  \wedge T_i }}{{\omega ^n }}} \right|}  \le& C\sum\limits_{i = 0}^{k - 2} \sum\limits_{1 \le j_1  <  \cdots  < j_i  \le n, \ne j,l} {\left| {u_j } \right|\left| {u_{\bar l} } \right|\left| {\lambda _{j_1 } \lambda _{j_2 }  \cdots \lambda _{j_i } } \right|}  \\
  \le& C \sum\limits_{i = 0}^{k - 2}\sum\limits_{j = 1}^{n} \sum\limits_{\begin{array}{*{20}c}
   {1 \le j_1< \cdots <j_i\le n}\\
   {j_l\ne j}  \\
\end{array}} {\left| {u_j } \right|^2 \left| {\lambda _{j_1 } \lambda _{j_2 }  \cdots \lambda _{j_i } } \right|}  \notag\\
  \le& C \sum\limits_{i = 0}^{k - 2}\sum\limits_{j = 1}^{n}\sigma _i \left( {\lambda \left| j \right.} \right)\left| {u_j } \right|^2  \notag\\
  =&{C\sum\limits_{i = 0}^{k - 2}{\frac{{\sqrt { - 1} \partial u \wedge \bar \partial u \wedge \omega _u ^i  \wedge \omega ^{n - i - 1} }}{{\omega ^n }}}},\notag
 \end{align}
where we have used the lemma $2.1$ in the last inequality.
\end{proof}

\subsection{Commutation formula of covariant derivatives}
In local complex coordinates $z_1,\cdots,z_n$, we have
\begin{align}
g_{i\bar j}=g(\frac{\p}{\p z^i},\frac{\p}{\p\bar z^j}),
\{g^{i\bar j}\}=\{g_{i\bar j}\}^{-1}
\end{align}
For the Chern connection $\nabla$ ,we denote the covariant derivatives as follows:
\begin{align}
u_i  = \nabla _{\frac{\partial }{{\partial z^i }}} u,u_{i\bar j }  = \nabla _{\frac{\partial }{{\partial\bar z^j  }}} \nabla _{\frac{\partial }{{\partial z^i }}} u,u_{i\bar j k}  = \nabla _{\frac{\partial }{{\partial z^k }}} \nabla _{\frac{\partial }{{\partial \bar z^j }}} \nabla _{\frac{\partial }{{\partial z^i }}} u
\end{align}
we use the following commutation formula for covariant derivatives on Hermitian manifolds which can be founded in \cite{TW5}:
\begin{align}\label{three}
u_{i\bar j l}  & = u_{l\bar j i}  - T_{li}^p u_{p\bar j }\\
u_{pi\bar j }  &= u_{p\bar j i}  + u_q {R_{i\bar j p}}^q \notag\\
u_{i\bar p \bar j } &= u_{i\bar j \bar p } - \bar {T_{jp}^q } u_{i\bar q }.\notag
\end{align}
\begin{align}\label{four}
u_{i\bar j l\bar m }  = u_{l\bar m i\bar j }  + u_{p\bar j } {R_{l\bar m i}}^p  - u_{p\bar m } {R_{i\bar j l}}^p  - T_{li}^p u_{p\bar m \bar j }  - \bar {T_{mj}^p } u_{l\bar p i}  - T_{li}^p \bar {T_{mj}^q } u_{p\bar q }
 \end{align}
 For the details we recommend the reader to the reference \cite{TW5}.

\section{ zero order estimate}
In this section we derive the zero order estimate by proving a Cherrier-type inequality and the lemmas in \cite{TW2}. Since the constant $b$ is in Theorem 1.1 satisfies \[|b|\le\sup|f|+C,\]where $C$ is a positive constant depending only on $(M, \omega)$. Thus, we will assume $b=0$ for convenience.
\begin{theorem} Let $u$ be a solution of Theorem 1.1. Then there exists
a constant $C$ depending only on $(M,\omega)$ and $\mathop {\sup }\limits_M |f|$
such that
\begin{center}
$\mathop {\sup }\limits_M |u|\leq C.$
\end{center}
\end{theorem}

Due to Tosatti-Weinkove's results, the zero order estimate can be reduced to derive a Cherrier-type inequality which was firstly proved in Cherrier's paper \cite{C}. For the Hessian equation, the analysis becomes a bit complicated in the lack of the positivity of $\omega_u$. Recently\footnote{The author independently proved the $C^0$ estimate before \cite{Sun} was posted on arXiv.}, Sun \cite{Sun} also proved the following lemma for $k=2$ and $k\ge3$ under some extra conditions.

 \begin{lemma}\label{Ch-Ineq}  There exist constants $p_0$ and $C$ depending only
 on $(M,\omega)$ such that for any $p\geq p_0$
\begin{align}
   \int_M |\p e^{-\frac{p}{2}u}|^2_g\omega^{n}
   \leq C p \int_M e^{-pu}\omega^{n} \notag
\end{align}
\end{lemma}
\begin{proof}
By the equation, we have
\begin{align*}
 {\omega_u^k} \wedge \omega^{n-k}-\omega^n=(e^f-1)\omega^n
 \leq { C_0 \omega^n} \notag,
\end{align*}
where $ C_0$ is a constant depending only on $f$.

On the other hand,
\begin{align}\label{C^0estimate}
\omega _u ^k  \wedge \omega ^{n - k}  - \omega ^n  = \left( {\omega _u ^k  - \omega ^k } \right) \wedge \omega ^{n - k}  = \sqrt { - 1} \partial \bar \partial u \wedge \alpha,
\end{align}
where $\alpha= {\sum\limits_{i = 1}^k {\omega _u ^{i - 1}  \wedge \omega ^{n - i} } } $.

Now multiply both sides in \eqref{C^0estimate} by $e^{ - pu}$ and integrate by parts ,
\begin{align}\label{C^0key}
 C_0 \int_M {e^{ - pu} \omega ^n }  \ge& \int_M {e^{ - pu} \sqrt { - 1} \partial \bar \partial u \wedge \alpha}  \\
  =& - \int_M {\partial e^{ - pu} \sqrt { - 1} \bar \partial u \wedge \alpha}
  + \int_M {e^{ - pu} \sqrt { - 1} \bar \partial u \wedge \partial  \alpha} \notag\\
  =& p\int_M {e^{ - pu} \sqrt { - 1} \partial u \wedge \bar \partial u \wedge\alpha}-\frac{1}{p}\int_M {\sqrt { - 1} \bar \partial e^{ - pu}  \wedge \partial  \alpha}   \notag\\
  =& p\int_M {e^{ - pu} \sqrt { - 1} \partial u \wedge \bar \partial u \wedge\alpha}+\frac{1}{p}\int_M { e^{ - pu}  \sqrt { - 1} \bar \partial  \partial  \alpha} \notag   \\
 :=&A+B\notag,
 \end{align}

where we  denote
\begin{align*}
A=&p\int_M {e^{ - pu} \sqrt { - 1} \partial u \wedge \bar \partial u \wedge \left( {\sum\limits_{i = 1}^k {\omega _u ^{i - 1}  \wedge \omega ^{n - i} } } \right)}\\
 B =&\frac{1}{p}\int_M { e^{ - pu}  \sqrt { - 1} \bar \partial  \partial  \alpha}.
\end{align*}

We will use  the term $A$ to control the terms $B$. Direct calculation gives
\begin{align*}
\partial\alpha= n\sum\limits_{i = 1}^{k - 1} {  \omega _u ^{i-1}  \wedge \omega ^{n - i - 1}  \wedge  \partial \omega  }+(n-k)\omega _u ^{k-1}  \wedge \omega ^{n - k - 1}  \wedge  \partial \omega
\end{align*}

\begin{align*}
\bar\partial\partial\alpha=&(n - k)( n - k - 1)\omega _u ^{k - 1}  \wedge \omega ^{n - k - 2}  \wedge \bar \partial \omega  \wedge \partial \omega+(n-k)\omega _u ^{k - 1}  \wedge \omega ^{n - k - 1}  \wedge \bar \partial \partial \omega\\
&+ (n-k)(n + k - 1) \omega _u ^{k - 2}  \wedge \omega ^{n - k - 1}  \wedge \bar \partial \omega  \wedge \partial \omega\\
&+ n( n - 1 )\sum\limits_{i = 0}^{k - 3}  \omega _u ^i  \wedge \omega ^{n - i - 3}  \wedge \bar \partial \omega  \wedge \partial \omega  + n\sum\limits_{i = 1}^{k - 2}  \omega _u ^i  \wedge \omega ^{n - i - 2}  \wedge \bar \partial \partial \omega
\end{align*}
Therefore, we have
\begin{align*}
B=& \frac{{\left( {n - k} \right)\left( {n - k - 1} \right)}}{p}\int_M {\sqrt { - 1} e^{ - pu} \omega _u ^{k - 1}  \wedge \omega ^{n - k - 2}  \wedge \bar \partial \omega  \wedge \partial \omega }\\
   &+ \frac{{\left( {n - k} \right)}}{p}\int_M {\sqrt { - 1} e^{ - pu} \omega _u ^{k - 1}  \wedge \omega ^{n - k - 1}  \wedge \bar \partial \partial \omega } \\
   &+\frac{{\left( {n + k - 1} \right)\left( {n - k} \right)}}{p}\int_M {\sqrt { - 1} e^{ - pu} \omega _u ^{k - 2}  \wedge \omega ^{n - k - 1}  \wedge \bar \partial \omega  \wedge \partial \omega }  \\
  &+ \frac{{n\left( {n - 1} \right)}}{p}\sum\limits_{i = 0}^{k - 3} {\int_M {\sqrt { - 1} e^{ - pu} \omega _u ^i  \wedge \omega ^{n - i - 3}  \wedge \bar \partial \omega  \wedge \partial \omega } } + \frac{n}{p}\sum\limits_{i = 1}^{k - 2} {\int_M {\sqrt { - 1} e^{ - pu} \omega _u ^i  \wedge \omega ^{n - i - 2}  \wedge \bar \partial \partial \omega } }   \\
\end{align*}
When $k=2$, the term $B$ just becomes
\begin{align}\label{k=2,0}
B=& \frac{{\left( {n - 2} \right)\left( {n -3} \right)}}{p}\int_M {\sqrt { - 1} e^{ - pu} \omega _u   \wedge \omega ^{n - 4}  \wedge \bar \partial \omega  \wedge \partial \omega }
   + \frac{{\left( {n - 2} \right)}}{p}\int_M {\sqrt { - 1} e^{ - pu} \omega _u   \wedge \omega ^{n - 3}  \wedge \bar \partial \partial \omega } \\
   &+\frac{{\left( {n + 1} \right)\left( {n - 2} \right)}}{p}\int_M {\sqrt { - 1} e^{ - pu}  \omega ^{n - 3}  \wedge \bar \partial \omega  \wedge \partial \omega }\notag\\
   =& \frac{{\left( {n - 2} \right)\left( {n -3} \right)}}{p}\int_M {\sqrt { - 1} e^{ - pu} \sqrt { - 1}\partial \bar\partial u   \wedge \omega ^{n - 4}  \wedge \bar \partial \omega  \wedge \partial \omega }
   + \frac{{\left( {n - 2} \right)}}{p}\int_M {\sqrt { - 1} e^{ - pu} \sqrt { - 1}\partial \bar\partial u \wedge \omega ^{n - 3}  \wedge \bar \partial \partial \omega }\notag\\
   &+\frac{{2( n-1)( n - 2 )}}{p}\int_M {\sqrt { - 1} e^{ - pu}  \omega ^{n - 3}  \wedge \bar \partial \omega  \wedge \partial \omega }+\frac{{\left( {n - 2} \right)}}{p}\int_M {\sqrt { - 1} e^{ - pu} \omega ^{n - 2}  \wedge \bar \partial \partial \omega }\notag\\
   \ge& \frac{{\left( {n - 2} \right)\left( {n -3} \right)}}{p}\int_M {\sqrt { - 1} e^{ - pu} \sqrt { - 1}\partial \bar\partial u    \wedge \omega ^{n - 4}  \wedge \bar \partial \omega  \wedge \partial \omega } \notag\\
   &+ \frac{{\left( {n - 2} \right)}}{p}\int_M {\sqrt { - 1} e^{ - pu} \sqrt { - 1}\partial \bar\partial u  \wedge \omega ^{n - 3}  \wedge \bar \partial \partial \omega }
   -\frac{C_1}{p}\int_M {e^{ - pu} \omega ^n }\notag
   \end{align}
We next use integration by parts again to deal with the first term and second term on the right hand side of the above equality. Indeed,
\begin{align}\label{k=2,1}
&\int_M {\sqrt { - 1} e^{ - pu} \sqrt { - 1}\partial \bar\partial u  \wedge \omega ^{n - 4}  \wedge \bar \partial \omega  \wedge \partial \omega }\notag\\
\qquad\qquad =&p\int_M {\sqrt { - 1} e^{ - pu} \sqrt { - 1}\partial u\wedge \bar\partial u    \wedge \omega ^{n - 4}  \wedge \bar \partial \omega  \wedge \partial \omega }+\int_M {\sqrt { - 1} e^{ - pu} \sqrt { - 1} \bar\partial u    \wedge \partial(\omega ^{n - 4}  \wedge \bar \partial \omega  \wedge \partial \omega )}\\
=&p\int_M {\sqrt { - 1} e^{ - pu} \sqrt { - 1}\partial u\wedge \bar\partial u    \wedge \omega ^{n - 4}  \wedge \bar \partial \omega  \wedge \partial \omega }+\frac{1}{p}\int_M {\sqrt { - 1} e^{ - pu} \sqrt { - 1}     \bar\partial  \partial(\omega ^{n - 4}  \wedge \bar \partial \omega  \wedge \partial \omega )}\notag\\
\ge& -pC_1 \int_M {e^{ - pu} \sqrt { - 1}\partial u\wedge \bar\partial u    \wedge \omega ^{n - 1}}-\frac{C_1}{p}\int_M {e^{ - pu} \omega ^n }\notag\\
\ge&-C_1A-\frac{C_1}{p}\int_M {e^{ - pu} \omega ^n }\notag
\end{align}
The similar calculation gives
\begin{align}\label{k=2,2}
\int_M {\sqrt { - 1} e^{ - pu} \sqrt { - 1}\partial \bar\partial u  \wedge \omega ^{n - 3}  \wedge \bar \partial \partial \omega }\ge-C_1A-\frac{C_1}{p}\int_M {e^{ - pu} \omega ^n }
\end{align}
Inserting \eqref{k=2,1} and \eqref{k=2,2} into \eqref{k=2,0}, we have

\[B\ge-\frac{C_1}{p}A-\frac{C_1}{p}\int_M {e^{ - pu} \omega ^n }\]

By \eqref{C^0key} and choosing $p_0=2C_1+1$, we obtain for $p\ge p_0$
\[\frac{A}{2}\le(1-\frac{C_1}{p})A\le(\frac{C_1}{p}+C_0)\int_M {e^{ - pu} \omega ^n }\le( C_0+1)\int_M {e^{ - pu} \omega ^n }\]
By \eqref{C^0basicfact} in the next page, we thus prove the lemma.

 \  For the general $k, 3\le k \le n$, we claim that there exist   constants $C_{1i}$ depending only on $n,k,\left(M,\omega\right)$ such that the following holds for $0\le i\le k-1$,
 \begin{align}\label{C^0claim}
\int_{\rm{M}} {e^{ - pu} \omega _u ^i  \wedge T_i }  \ge  -{p} C_{1i}\sum\limits_{j = 0}^{k - 2} {\int_M {e^{ - pu} \sqrt { - 1} \partial u \wedge \bar \partial u \wedge \omega _u ^j  \wedge \omega ^{n - j - 1} } }  -C_{1i}\int_M {e^{ - pu} \omega ^n }
\end{align}
,where $T_i$ is defined as the combinations of $\omega,\partial \omega, \partial\bar\partial \omega$, more precisely
\begin{align*}
T_i  = \sum\limits_{0 \le 3p + 2q \le n - i} {\omega ^{n - i - 3p - 2q}  \wedge (\sqrt { - 1})^p\left( {\partial \omega } \right)^p  \wedge \left( {\bar \partial \omega } \right)^p  \wedge (\sqrt { - 1})^q \left( {\partial \bar \partial \omega } \right)^q }
\end{align*}
We use the claim \eqref{C^0claim} to prove the lemma
 \begin{align*}
 B
 \ge&  - C_1 \sum\limits_{i = 2}^k {\int_M {e^{ - pu} \sqrt { - 1} \partial u \wedge \bar \partial u \wedge \omega _u ^{k - i}  \wedge \omega ^{n + i - k - 1} } }  - \frac{{C_1 }}{p}\int_M {e^{ - pu} \omega ^n }  \\
  \ge&  - \frac{C_1}{p} A  - \frac{{C_1 }}{p}\int_M {e^{ - pu} \omega ^n }
 \end{align*}

Thus we have
\begin{align*}
(1-\frac{C_1}{p})A\le(\frac{C_1}{p}+C_0)\int_M {e^{ - pu} \omega ^n }
\end{align*}

Now we choose $p_0=2C_1+1$, then for any $p\geq p_0$,
\begin{align}
&p^2\int_M e^{-pu}\sqrt{-1}\p u\wedge\bar\p u\wedge\omega^{n-1}
\leq 2p(C_0+1)\int_M e^{-pu}\omega^n  \notag
\end{align}

Therefore we have
\begin{align}\label{C^0basicfact}
\int_M |\p e^{-\frac{p}{2}u}|^2_g\omega^{n}&=\frac{np^2}{4}\int_M e^{-pu}\sqrt{-1}\p u\wedge\bar\p u\wedge \omega^{n-1} \\
&\leq\frac{np(C_0+1)}{2}\int_M e^{-pu}\omega^{n}
=p C\int_M e^{-pu}\omega^{n}\notag
\end{align}

Now, we prove the claim \eqref{C^0claim} by inductive argument.

\ When $i=1$, we have
\begin{align*}
 \int_{\rm{M}} {e^{ - pu} \omega _u  \wedge T_{\rm{1}} }  = & \int_{\rm{M}} {e^{ - pu} \omega  \wedge T_{\rm{1}} } {\rm{ + }}\int_{\rm{M}} {e^{ - pu} \sqrt { - 1} \partial \bar \partial u \wedge T_{\rm{1}} }  \\
 =&\int_{\rm{M}} {e^{ - pu} \omega  \wedge T_{\rm{1}} }  - \int_{\rm{M}} {\partial e^{ - pu}  \wedge \sqrt { - 1} \bar \partial u \wedge T_{\rm{1}} }  + \int_{\rm{M}} {e^{ - pu} \sqrt { - 1} \bar \partial u \wedge \partial T_{\rm{1}} }  \\
 =& \int_{\rm{M}} {e^{ - pu} \omega  \wedge T_{\rm{1}} }  + p\int_{\rm{M}} {e^{ - pu} \sqrt { - 1} \partial u \wedge \bar \partial u \wedge T_{\rm{1}} }  - \frac{1}{p}\int_{\rm{M}} {\sqrt { - 1} \bar \partial e^{ - pu}  \wedge \partial T_{\rm{1}} }  \\
  =& p\int_{\rm{M}} {e^{ - pu} \sqrt { - 1} \partial u \wedge \bar \partial u \wedge T_{\rm{1}} }  + \int_{\rm{M}} {e^{ - pu} \omega  \wedge T_{\rm{1}} }  - \frac{1}{p}\int_{\rm{M}} {e^{ - pu}  \wedge \sqrt { - 1} \partial \bar \partial T_{\rm{1}} }  \\
  &\ge  - C_1 p\int_{\rm{M}} {e^{ - pu} \sqrt { - 1} \partial u \wedge \bar \partial u \wedge T_{\rm{1}} }  - C_1 \int_{\rm{M}} {e^{ - pu} \omega ^n }
 \end{align*}

 Suppose that the claim is true for $l\le i-1$, we will prove that the claim is also true for $l=i$. Indeed,
\begin{align*}
 \int_{\rm{M}} {e^{ - pu} \omega _u ^i  \wedge T_i }  =& \int_{\rm{M}} {e^{ - pu} \omega _u ^{i - 1}  \wedge \omega  \wedge T_i }  + \int_{\rm{M}} {e^{ - pu} \sqrt { - 1} \partial \bar \partial u \wedge \omega _u ^{i - 1}  \wedge T_i }  \\
  =& \int_{\rm{M}} {e^{ - pu} \omega _u ^{i - 1}  \wedge \omega  \wedge T_i }  + p\int_{\rm{M}} {e^{ - pu} \sqrt { - 1} \partial u \wedge \bar \partial u \wedge \omega _u ^{i - 1}  \wedge T_i } \\
   &+ \int_{\rm{M}} {e^{ - pu} \bar \partial u \wedge \sqrt { - 1} \partial \left( {\omega _u ^{i - 1}  \wedge T_i } \right)}  \\
 : =& A_{i,1}  + A_{i,2}  + A_{i,3}
 \end{align*}
By the induction ,
\begin{align*}
 A_{i,1}  =& \int_{\rm{M}} {e^{ - pu} \omega _u ^{i - 1}  \wedge \omega  \wedge T_i }  \\
  \ge&  - pC_{1i} \left( {n,k,\omega } \right)\sum\limits_{j = 0}^{k - 2} {\int_M {e^{ - pu} \sqrt { - 1} \partial u \wedge \bar \partial u \wedge \omega _u ^j  \wedge \omega ^{n - j - 1} } }  - C_{1i} \left( {n,k,\omega } \right)\int_M {e^{ - pu} \omega ^n }
 \end{align*}

By the inequality \eqref{lemma2.2} in lemma \ref{inequality}, we have
\begin{align}\label{C^0keyinequality}
A_{i,2}  = p\int_{\rm{M}} {e^{ - pu} \sqrt { - 1} \partial u \wedge \bar \partial u \wedge \omega _u ^{i - 1}  \wedge T_i }  \ge  -p C_{2i} \int_{\rm{M}} {e^{ - pu} \sqrt { - 1} \partial u \wedge \bar \partial u \wedge \omega _u ^{i - 1}  \wedge \omega ^{n - i} }
\end{align}
Now we deal with the term $A_{i,3}$,
\begin{align*}
 A_{i,3}  =& \int_{\rm{M}} {e^{ - pu} \bar \partial u \wedge \sqrt { - 1} \partial \left( {\omega _u ^{i - 1}  \wedge T_i } \right)}= \frac{1}{p}\int_{\rm{M}} {e^{ - pu}   \sqrt { - 1} \bar \partial\partial \left( {\omega _u ^{i - 1}  \wedge T_i } \right)} \\
 =& \frac{{\left( {i - 1} \right)\left( {i - 2} \right)}}{p}\int_{\rm{M}} {e^{ - pu} \sqrt { - 1} \omega _u ^{i - 3}  \wedge \bar \partial \omega  \wedge \partial \omega  \wedge T_i } {\rm{ + }}\frac{{i - 1}}{p}\int_{\rm{M}} {e^{ - pu} \omega _u ^{i - 2}  \wedge \sqrt { - 1} \bar \partial \left( {\partial \omega  \wedge T_i } \right)}  \\
  &+ \frac{{i - {\rm{1}}}}{p}\int_{\rm{M}} {e^{ - pu} \omega _u ^{i - 2}  \wedge } \sqrt { - 1} \bar \partial \omega  \wedge \partial T_i  - \frac{1}{p}\int_{\rm{M}} {e^{ - pu} \omega _u ^{i - {\rm{1}}}  \wedge \sqrt { - 1} \partial \bar \partial T_i }  \\
  =& \frac{{\left( {i - 1} \right)\left( {i - 2} \right)}}{p}\int_{\rm{M}} {e^{ - pu} \sqrt { - 1} \omega _u ^{i - 3}  \wedge \bar \partial \omega  \wedge \partial \omega  \wedge T_i }  \\
  &+ \frac{{i - 1}}{p}\int_{\rm{M}} {e^{ - pu} \omega _u ^{i - 2}  \wedge \left[ {\sqrt { - 1} \bar \partial \left( {\partial \omega  \wedge T_i } \right) + \sqrt { - 1} \bar \partial \omega  \wedge \partial T_i } \right]}  - \frac{1}{p}\int_{\rm{M}} {e^{ - pu} \omega _u ^{i - {\rm{1}}}  \wedge \sqrt { - 1} \partial \bar \partial T_i }  \\
  \ge&  - pC_{3i} \sum\limits_{j = 0}^{k - 2} {\int_M {e^{ - pu} \sqrt { - 1} \partial u \wedge \bar \partial u \wedge \omega _u ^j  \wedge \omega ^{n - j - 1} } }  - C_{3i} \left( {n,k,\omega } \right)\int_M {e^{ - pu} \omega ^n } . \\
 \end{align*}
For the last inequality, we have used the induction.
\end{proof}

\section{second order estimate}
In this section we use the auxiliary function in \cite{TW5} which is modified by the auxiliary function in \cite{HMW} to derive the second order estimate of the form \eqref{C2estimate}. The difficult part arises from the third order derivatives'   Locally the equation is
\begin{equation}\label{loceq}
\sigma_k(\omega_u)=e^{f}.
\end{equation}

\begin{theorem}
There exists a uniform constant $C$ depending only on $(M,\omega)$ and $f$ such that
\begin{equation}\label{4.2}
  \max|\p\bar\p u|_g\leq C(1+\max|\nabla u|^2_g)
\end{equation}
\end{theorem}

\begin{proof}
Denote $w_{i\bar{j}} = g_{i\bar{j}} + u_{i\bar{j}}$ and let $\xi \in T^{1,0}M,\ |\xi|_g^2=1$.

We use the auxiliary function which is similar to the one in \cite{TW5}
  \[   H({x,\xi})=\log(w_{k\bar l}\xi^k\bar\xi^l)+c_0\log(g^{k\bar l}w_{p\bar l}  w_{k\bar q}\xi^p\bar\xi^q)+\varphi(|\triangledown u|^2_g)+\psi(u), \]
where $\varphi,\psi$ are given by
  \[\begin{split}
   \varphi \left( s \right) &=  - \frac{1}{2}\log \left( {1 - \frac{s}{{2K}}} \right),\quad 0 \le s \le K - 1,\\
   \psi \left( t \right) &=  - A\log \left( {1 + \frac{t}{{2L}}} \right),\quad  - L + 1 \le t \le 0,\end{split}\]
for
\[ K: = \mathop {\sup }\limits_M |\nabla u|_g^2  + 1, \quad
   L= \mathop {\sup }\limits_M \left| u \right| + 1,\  A: = 2L(C_0  + 1) ,\]
 where $A_0$ is a constant to be determined later. $c_0$ is a small positive constant depending only on $n$ and will be determined later.
By \cite{HMW}, we have
\begin{align}
   \frac{1}{2 K} & \ge \varphi' \ge \frac{1}{4 K} > 0, \ \varphi'' = 2 \big(\varphi')^2 > 0. \\
  \frac{A}{L} & \ge - \psi' \ge \frac{A}{2L} = C_0 +1,
  \ \psi''  \ge \frac{2\varepsilon_0}{1 - \varepsilon_0} (\psi')^2,\  \quad \textup{for all $\varepsilon_0 \le \frac{1}{2A + 1}$}.
\end{align}
These inequalities will be used below.

\ Suppose $H({x,\xi})$ attains its maximum at the point $x_0$ in the direction $\xi_0$, then we choose local coordinates
$\{\frac{\p}{\p z^1},\cdots,\frac{\p}{\p z^n}\}$ near $x_0$ such that
  \[ \begin{split}
  g_{i\bar j}(x_0)&=\delta_{ij},
  u_{i\bar j}=u_{i\bar i}(x_0)\delta_{ij},\\
  \lambda_i=w_{i\bar i}(x_0)&=1+u_{i\bar i}(x_0) \text{ with }  \lambda_1\geq\cdots\geq\lambda_n.
 \end{split}\]
We will prove  that
  \[ H({x_0,\xi})\leq H({x_0,\frac{\p}{\p z^1}})\quad \forall \xi \in T^{1,0}M,\ |\xi|_g^2=1,
  \sum\limits_{i,j}{w_{i\bar j}(x_0)\xi^i\xi^{\bar j}}>0 \]
by choosing $c_0$ small enough. In fact, at $x_0$ we have
  \[ \log(w_{k\bar l}\xi^k\bar\xi^l)+c_0\log(g^{k\bar l}w_{p\bar l}
   w_{k\bar q}\xi^p\bar\xi^q)= \log (\sum\limits_{k = 1}^n {w_{k\bar k } \left| {\xi ^k } \right|^2 } ) + c_0 \log (\sum\limits_{k = 1}^n {\left| {w_{k\bar k } } \right|^2 \left| {\xi ^k } \right|^2 } )\notag
   \]
If $w_{n\bar n}\ge-w_{1\bar 1}$ which is always satisfied when $n\le3$ , we have $w_{i\bar i}^2\leq w_{1\bar 1}$. Thus we have $H({x_0,\xi})\leq H({x_0,\frac{\p}{\p z^1}})$.

Now we suppose that $w_{n\bar n }< - w_{1\bar 1}$, thus we have $n\ge4$. Let $i_0$ be the smallest integer satisfying $w_{i\bar i }<-w_{1\bar 1 }$, then $i_0\ge k+1$. By $|w_{i\bar i }|<(n-2)w_{1\bar1}$ we have
  \[\begin{split}
     &\log (\sum\limits_{i=1}^n{w_{i\bar i}\left|{\xi ^i}\right|^2})+c_0\log(\sum\limits_{i=1}^n{\left|{w_{i\bar i}}\right|^2\left|{\xi^i}\right|^2})\\
  \le&\log w_{1\bar 1}(\sum\limits_{i=1}^{i_0-1}{\left|{\xi^i}\right|^2-\sum\limits_{i=i_0}^n{\left|{\xi^i}\right|^2}})
       +c_0\log(w_{1\bar 1}^2\sum\limits_{i=1}^{i_0-1}{\left|{\xi^i}\right|^2+(n-2)^2w_{1\bar1}^2\sum\limits_{i=1}^{i_0-1}{\left|{\xi^i}\right|^2}})\\
    =&\log w_{1\bar1}(1-2t)+c_0\log w_{1\bar1}^2(1-t+(n-2)^2t):=h(t),
  \end{split}\]
where $t=\sum\limits_{i = i_0 }^n {\left| {\xi ^i } \right|^2} \in(0,\frac{1}{2})$.

By choosing $c_0=\frac{2}{(n-2)^2-1}$, we have $h'(t)\leq0$,
thus \[h(t)\leq h(0)=\log (w_{1\bar 1 } ) + c_0 \log w_{1\bar 1 }^2.\]
Consequently, we have proved
  \[H({x_0,\xi})\leq H({x_0,\frac{\p}{\p z^1}}),\ \text{for}\  \forall \xi \in T^{1,0}M,
  | \xi |_g^2  = 1,\sum\limits_{i,j} {\eta _{i\bar j } (x_0 )\xi ^i \xi ^{\bar j } } > 0,\]
   by choosing $c_0=\frac{2}{(n-2)^2-1}$ when $n\ge 4$ and $c_0=1$ when $n\le 3$.

We extend $\xi_0$ near $x_0$ as
  \[\xi_0=(g_{1\bar1})^{\frac{1}{2}}\frac{\p}{\p z^1}.\]
Consider the function
  \[ Q(x)=H({x,\xi_0})=\log(g_{1\bar1}^{-1}w_{1\bar l})+c_0\log(g_{1\bar1}^{-1}g^{k\bar l}w_{1\bar l}
 w_{k\bar 1})+\varphi(|\triangledown u|^2_g)+\psi(u).\]
We will calculate $F^{i\bar j}Q_{i\bar j}\ at \ x_0$ to get the estimate, all the calculations are taken at $x_0$.
For simplicity, we denote $\xi=\xi_0$ in the following. By $\langle\xi,\bar\xi\rangle_g=| \xi |_g^2  = 1$, differentiating both sides, we obtain at $x_0$
\begin{align} \label{4.3}
0=\frac{\p}{\p z^i}\langle\xi,\bar\xi\rangle_g&=\langle\nabla_\frac{\p}{\p z^i}\xi,\bar\xi\rangle_g +\langle\xi,\nabla_\frac{\p}{\p z^i}\bar\xi\rangle_g\notag\\
&=\langle {\xi^k}_{,i}\frac{\p}{\p z^k},\bar{\xi^l}\frac{\p}{\bar\p z^l}
\rangle_g +\langle \xi^k\frac{\p}{\p z^k},\bar{\xi^l}_{,i}\frac{\p}{\bar\p z^l} \rangle_g \notag\\
&=g_{k\bar l}{\xi^k}_{,i}\bar{\xi^l} + g_{k\bar l}{\xi^k}{\bar{\xi^l}}_{,i}\notag\\
&={\xi^1}_{,i}+{\bar{\xi^1}}_{,i}.
\end{align}

We also have the basic formula for $\xi\in T^{1,0}M$:
\begin{align}\label{4.4}
{\bar {{\xi ^k}} _{,i}} = \frac{{\partial \bar {{\xi ^k}} }}{{\partial {z^i}}} = \overline {\frac{{\partial {\xi ^k}}}{{\partial \bar {{z^i}} }}}  = \overline {{\xi ^k}_{,\bar i }} ,\notag\\
{\xi ^k}_{,\bar i } = \frac{{\partial {\xi ^k}}}{{\partial \bar {{z^i}} }} = \overline {\frac{{\partial \bar {{\xi ^k}} }}{{\partial {z^i}}}}  = \overline {{{\bar {{\xi ^k}} }_{,i}}} \notag\\
{\overline {{\xi ^k}} _{,i}} = \frac{{\partial \bar {{\xi ^k}} }}{{\partial {z^i}}} = \overline {\frac{{\partial {\xi ^k}}}{{\partial \bar {{z^i}} }}}  = \overline {{\xi ^k}_{,\bar i }} ,\notag\\
{\xi ^k}_{,\bar i } = \frac{{\partial {\xi ^k}}}{{\partial \bar {z^i} }} = \overline {\frac{{\partial \bar {{\xi ^k}} }}{{\partial {z^i}}}}  = \overline {{{\overline {{\xi ^k}} }_{,i}}}
\end{align}
Direct calculations give
\begin{align*}
{Q_i} =& \frac{{{{\left( {{w_{k\bar l}}{\xi ^k}\bar {{\xi ^l}} } \right)}_i}}}{{{w_{k\bar l}}{\xi ^k}\bar {{\xi ^l}} }} + {c_0}\frac{{{{\left( {{g^{p\bar q}}{w_{k\bar q}}{w_{p\bar l}}{\xi ^k}\bar {{\xi ^l}} } \right)}_i}}}{{{g^{p\bar q}}{w_{k\bar q}}{w_{p\bar l}}{\xi ^k}\bar {{\xi ^l}} }} + {\varphi_i} + {\psi _i}\\
{Q_{i\bar i}} = &\frac{{{{\left( {{w_{k\bar l}}{\xi ^k}\bar {{\xi ^l}} } \right)}_i}_{\bar i}}}{{{w_{k\bar l}}{\xi ^k}\bar {{\xi ^l}} }} - \frac{{{{\left( {{w_{k\bar l}}{\xi ^k}\bar {{\xi ^l}} } \right)}_i}{{\left( {{w_{k\bar l}}{\xi ^k}\bar {{\xi ^l}} } \right)}_{\bar i}}}}{{{{\left( {{w_{k\bar l}}{\xi ^k}\bar {{\xi ^l}} } \right)}^2}}} + {c_0}\frac{{{{\left( {{g^{p\bar q}}{w_{k\bar q}}{w_{p\bar l}}{\xi ^k}\bar {{\xi ^l}} } \right)}_{i\bar i}}}}{{{g^{p\bar q}}{w_{k\bar q}}{w_{p\bar l}}{\xi ^k}\bar {{\xi ^l}} }} \notag\\
 &- {c_0}\frac{{{{\left( {{g^{p\bar q}}{w_{k\bar q}}{w_{p\bar l}}{\xi ^k}\bar {{\xi ^l}} } \right)}_i}{{\left( {{g^{p\bar q}}{w_{k\bar q}}{w_{p\bar l}}{\xi ^k}\bar {{\xi ^l}} } \right)}_{\bar i}}}}{{{{\left( {{g^{p\bar q}}{w_{k\bar q}}{w_{p\bar l}}{\xi ^k}\bar {{\xi ^l}} } \right)}^2}}} + {\varphi _{i\bar i}} + {\psi _{i\bar i}}
\end{align*}
Next, we will simplify $Q_i$ and $Q_{i\bar i}$.

 By $ \eqref{4.3}$, we have
\begin{align*}
{\left( {{w_{k\bar l}}{\xi ^k}\bar {{\xi ^l}} } \right)_i} &= {w_{k\bar l,i}}{\xi ^k}\bar {{\xi ^l}}  + {w_{k\bar l}}{\xi ^k}_{,i}\bar {{\xi ^l}}  + {w_{k\bar l}}{\xi ^k}{\bar {{\xi ^l}} _{,i}} \\
&= {w_{1\bar 1,i}} + {w_{1\bar 1}}\left( {{\xi ^1}_{,i} + {{\bar {{\xi ^1}} }_{,i}}} \right) \\
&= {w_{1\bar 1i}},
\end{align*}
Thus we have
\begin{align*}
{\left( {{g^{p\bar q}}{w_{k\bar q}}{w_{p\bar l}}{\xi ^k}\bar {{\xi ^l}} } \right)_i} &= {g^{p\bar q}}{w_{k\bar qi}}{w_{p\bar l}}{\xi ^k}\bar {{\xi ^l}}  + {g^{p\bar q}}{w_{k\bar q}}{w_{p\bar li}}{\xi ^k}\bar {{\xi ^l}}  + {g^{p\bar q}}{w_{k\bar q}}{w_{p\bar l}}{\xi ^k}_i\bar {{\xi ^l}}  + {g^{p\bar q}}{w_{k\bar q}}{w_{p\bar l}}{\xi ^k}{\bar {{\xi ^l}} _{,i}}\\
 &= {w_{1\bar 1}}\left( {{w_{1\bar 1i}} + {w_{1\bar 1i}}} \right) + {w_{1\bar 1}}^2\left( {{\xi ^1}_{,i} + {{\bar {{\xi ^1}} }_{,i}}} \right) \\
 &= 2{w_{1\bar 1}}{w_{1\bar 1i}}.
 \end{align*}
Therefore, we obtain the simplified formula for the term $Q_i$ at $x_0$.
   \begin{equation}\label{firstordercondition}
   {Q_i} = \frac{{{w_{1\bar 1i}}}}{{{w_{1\bar 1}}}} + {c_0}\frac{{2{w_{1\bar 1i}}}}{{{w_{1\bar 1}}}} + {\varphi _i} + {\psi _i} = (1 + 2{c_0})\frac{{{w_{1\bar 1i}}}}{{{w_{1\bar 1}}}} + {\varphi _i} + {\psi _i} = 0
   \end{equation}

Similar calculations give
\begin{align*}
{\left( {{w_{k\bar l}}{\xi ^k}\bar {{\xi ^l}} } \right)_i}_{\bar i} =& {\left[ {{w_{k\bar li}}{\xi ^k}\bar {{\xi ^l}}  + {w_{k\bar l}}\left( {{\xi ^k}_i\bar {{\xi ^l}}  + {\xi ^k}{{\bar {{\xi ^l}} }_i}} \right)} \right]_{\bar i}}\\
=& {w_{k\bar li\bar i}}{\xi ^k}\bar {{\xi ^l}}  + {w_{k\bar li}}\left( {{\xi ^k}_{\bar i}\bar {{\xi ^l}}  + {\xi ^k}{{\bar {{\xi ^l}} }_{\bar i}}} \right) + {w_{k\bar l\bar i}}\left( {{\xi ^k}_i\bar {{\xi ^l}}  + {\xi ^k}{{\bar {{\xi ^l}} }_i}} \right)\\
 &+ {w_{k\bar l}}\left( {{\xi ^k}_{i\bar i}\bar {{\xi ^l}}  + {\xi ^k}_{\bar i}{{\bar {{\xi ^l}} }_i} + {\xi ^k}_i{{\bar {{\xi ^l}} }_{\bar i}} + {\xi ^k}{{\bar {{\xi ^l}} }_{i\bar i}}} \right)\\
=& {w_{1\bar 1i\bar i}} + {w_{k\bar 1i}}{\xi ^k}_{\bar i} + {w_{1\bar li}}{\bar {{\xi ^l}} _{\bar i}} + {w_{k\bar 1\bar i}}{\xi ^k}_i + {w_{1\bar l\bar i}}{\bar {{\xi ^l}} _i}\\
 &+ {w_{1\bar 1}}({\xi ^1}_{i\bar i} + {\bar {{\xi ^1}} _{i\bar i}}) + {w_{k\bar k}}({\xi ^k}_{\bar i}{\bar {{\xi ^k}} _i} + {\xi ^k}_i{\bar {{\xi ^k}} _{\bar i}})\\
=& {w_{1\bar 1i\bar i}} + 2\sum\limits_{k \ne 1} {{\mathop{\rm Re}\nolimits} ({w_{k\bar 1i}}{\xi ^k}_{\bar i} + {w_{1\bar ki}}\bar {{\xi ^k}_i} )}  + {w_{1\bar 1}}({\xi ^1}_{i\bar i} + {\bar {{\xi ^1}} _{i\bar i}}) + {w_{k\bar k}}({\left| {{\xi ^k}_{\bar i}} \right|^2} + {\left| {{\xi ^k}_i} \right|^2}).
\end{align*}
The last equality holds because we have used $\eqref{4.2}$ and $\eqref{4.3}$ and the fact
\begin{align*}
{w_{k\bar 1i}}{\xi ^k}_{\bar i} + {w_{1\bar l\bar i}}{\bar {{\xi ^l}} _i} &= 2{\mathop{\rm Re}\nolimits} ({w_{k\bar 1i}}{\xi ^k}_{\bar i}),\\
{w_{1\bar li}}{\bar {{\xi ^l}} _{\bar i}} + {w_{k\bar 1\bar i}}{\xi ^k}_i & = 2{\mathop{\rm Re}\nolimits} ({w_{1\bar ki}}\bar {{\xi ^k}_i} ).\\
\end{align*}

We can also calculate 
 \begin{align*}
{\left( {{g^{p\bar q}}{w_{k\bar q}}{w_{p\bar l}}{\xi ^k}\bar {{\xi ^l}} } \right)_{i\bar i}}&={g^{p\bar q}}{\left( {{w_{k\bar qi}}{w_{p\bar l}}{\xi ^k}\bar {{\xi ^l}}  + {w_{k\bar q}}{w_{p\bar li}}{\xi ^k}\bar {{\xi ^l}}  + {w_{k\bar q}}{w_{p\bar l}}{\xi ^k}_i\bar {{\xi ^l}}  + {w_{k\bar q}}{w_{p\bar l}}{\xi ^k}{{\bar {{\xi ^l}} }_i}} \right)_{\bar i}}\\
=& {g^{p\bar q}}\left( {{w_{k\bar qi\bar i}}{w_{p\bar l}}{\xi ^k}\bar {{\xi ^l}}  + {w_{k\bar qi}}{w_{p\bar l\bar i}}{\xi ^k}\bar {{\xi ^l}}  + {w_{k\bar qi}}{w_{p\bar l}}{\xi ^k}_{\bar i}\bar {{\xi ^l}}  + {w_{k\bar qi}}{w_{p\bar l}}{\xi ^k}{{\bar {{\xi ^l}} }_{\bar i}}} \right)\\
 &+ {g^{p\bar q}}\left( {{w_{k\bar q\bar i}}{w_{p\bar li}}{\xi ^k}\bar {{\xi ^l}}  + {w_{k\bar q}}{w_{p\bar li\bar i}}{\xi ^k}\bar {{\xi ^l}}  + {w_{k\bar q}}{w_{p\bar li}}{\xi ^k}_{\bar i}\bar {{\xi ^l}}  + {w_{k\bar q}}{w_{p\bar li}}{\xi ^k}{{\bar {{\xi ^l}} }_{\bar i}}} \right)\\
 &+ {g^{p\bar q}}\left( {{w_{k\bar q\bar i}}{w_{p\bar l}}{\xi ^k}_i\bar {{\xi ^l}}  + {w_{k\bar q}}{w_{p\bar l\bar i}}{\xi ^k}_i\bar {{\xi ^l}}  + {w_{k\bar q}}{w_{p\bar l}}{\xi ^k}_{i\bar i}\bar {{\xi ^l}}  + {w_{k\bar q}}{w_{p\bar l}}{\xi ^k}_i{{\bar {{\xi ^l}} }_{\bar i}}} \right)\\
 &+ {g^{p\bar q}}\left( {{w_{k\bar q\bar i}}{w_{p\bar l}}{\xi ^k}{{\bar {{\xi ^l}} }_i} + {w_{k\bar q}}{w_{p\bar l\bar i}}{\xi ^k}{{\bar {{\xi ^l}} }_i} + {w_{k\bar q}}{w_{p\bar l}}{\xi ^k}_{\bar i}{{\bar {{\xi ^l}} }_i} + {w_{k\bar q}}{w_{p\bar l}}{\xi ^k}{{\bar {{\xi ^l}} }_{i\bar i}}} \right)\\
=& {w_{1\bar 1i\bar i}}{w_{1\bar 1}} + {w_{1\bar pi}}{w_{p\bar 1\bar i}} + {w_{k\bar 1i}}{w_{1\bar 1}}{\xi ^k}_{\bar i} + {w_{1\bar pi}}{w_{p\bar p}}{\bar {{\xi ^p}} _{\bar i}}\\
 &+ {w_{1\bar p\bar i}}{w_{p\bar 1i}} + {w_{1\bar 1}}{w_{1\bar 1i\bar i}} + {w_{p\bar p}}{w_{p\bar 1i}}{\xi ^p}_{\bar i} + {w_{1\bar 1}}{w_{1\bar li}}{\bar {{\xi ^l}} _{\bar i}}\\
 &+ {w_{k\bar 1\bar i}}{w_{1\bar 1}}{\xi ^k}_i + {w_{p\bar p}}{w_{p\bar 1\bar i}}{\xi ^p}_i{\rm{ + }}{w_{1\bar 1}}^2{\xi ^1}_{i\bar i}{\rm{ + }}{w_{p\bar p}}^2{\xi ^{\rm{p}}}_i{\bar {{\xi ^p}} _{\bar i}}\\
 &+ {w_{1\bar p\bar i}}{w_{p\bar p}}{\bar {{\xi ^p}} _i} + {w_{1\bar 1}}{w_{1\bar l\bar i}}{\bar {{\xi ^l}} _i} + {w_{p\bar p}}^2{\xi ^{\rm{p}}}_{\bar i}{\bar {{\xi ^p}} _i} + {w_{1\bar 1}}^2{\bar {{\xi ^1}} _{i\bar i}}\\
=& 2{w_{1\bar 1}}{w_{1\bar 1i\bar i}} + {\left| {{w_{1\bar pi}}} \right|^2} + {\left| {{w_{1\bar p\bar i}}} \right|^2} + 2{w_{1\bar 1}}{\mathop{\rm Re}\nolimits} ({w_{p\bar 1i}}{\xi ^p}_{\bar i} + {w_{p\bar 1\bar i}}{\xi ^p}_i)\\
 &+ 2{w_{p\bar p}}{\mathop{\rm Re}\nolimits} ({w_{1\bar pi}}{\bar {{\xi ^p}} _{\bar i}} + {w_{p\bar 1i}}{\xi ^p}_{\bar i}) + {w_{p\bar p}}^2\left( {{{\left| {{\xi ^{\rm{p}}}_i} \right|}^2} + {{\left| {{\xi ^{\rm{p}}}_{\bar i}} \right|}^2}} \right) + {w_{1\bar 1}}^2({\xi ^1}_{i\bar i}{\rm{ + }}{\bar {{\xi ^1}} _{i\bar i}})
 \end{align*}
Therefore we have simplify
$Q_{i\bar i}$ at $x_0$ as follows
\begin{align*}
{Q_{i\bar i}}=& (1 + 2{c_0})\frac{{{w_{1\bar 1i\bar i}}}}{{{w_{1\bar 1}}}} + \frac{{{{\rm{c}}_0}}}{{{w_{1\bar 1}}^2}}\sum\limits_{p \ne 1} {\left( {{{\left| {{w_{1\bar pi}}} \right|}^2} + {{\left| {{w_{1\bar p\bar i}}} \right|}^2}} \right)} \\
 &- (1 + 2{c_0})\frac{{{{\left| {{w_{1\bar 1i}}} \right|}^2}}}{{{w_{1\bar 1}}^2}}{\rm{ + }}{\left( { *  * } \right)_{i\bar i}} + {\varphi _{i\bar i}} + {\psi _{i\bar i}} ,
\end{align*}
where $\left( { *  * } \right)_{i\bar i}$ is given by
\begin{align*}\
{\left( { *  * } \right)_{i\bar i}} =& \frac{2}{{{w_{1\bar 1}}}}\sum\limits_{k \ne 1} {{\mathop{\rm Re}\nolimits} ({w_{k\bar 1i}}{\xi ^k}_{\bar i} + {w_{1\bar ki}}\bar {{\xi ^k}_i} )}  + {\xi ^1}_{i\bar i} + {\bar {{\xi ^1}} _{i\bar i}} + \frac{{{w_{k\bar k}}}}{{{w_{1\bar 1}}}}({\left| {{\xi ^k}_{\bar i}} \right|^2} + {\left| {{\xi ^k}_i} \right|^2})\\
 &+ \frac{{{\rm{2}}{{\rm{c}}_0}}}{{{w_{1\bar 1}}}}\sum\limits_{p \ne 1} {{\mathop{\rm Re}\nolimits} ({w_{p\bar 1i}}{\xi ^p}_{\bar i} + {w_{p\bar 1\bar i}}{\xi ^p}_i)}  + \sum\limits_{p \ne 1} {\frac{{{\rm{2}}{{\rm{c}}_0}{w_{p\bar p}}}}{{{w_{1\bar 1}}^2}}{\mathop{\rm Re}\nolimits} ({w_{1\bar pi}}{{\bar {{\xi ^p}} }_{\bar i}} + {w_{p\bar 1i}}{\xi ^p}_{\bar i})} \\
 &+ \frac{{{\rm{2}}{{\rm{c}}_0}{w_{p\bar p}}^2}}{{{w_{1\bar 1}}^2}}\left( {{{\left| {{\xi ^{\rm{p}}}_i} \right|}^2} + {{\left| {{\xi ^{\rm{p}}}_{\bar i}} \right|}^2}} \right) + {c_0}({\xi ^1}_{i\bar i}{\rm{ + }}{\bar {{\xi ^1}} _{i\bar i}}).
\end{align*}
For this term $\left( { *  * } \right)_{i\bar i}$, we have the  following estimate
\begin{align*}\
{\left( { *  * } \right)_{i\bar i}} \ge  - \frac{{{{\rm{c}}_0}}}{{2{w_{1\bar 1}}^2}}\sum\limits_{p \ne 1} {\left( {{{\left| {{w_{1\bar pi}}} \right|}^2} + {{\left| {{w_{1\bar p\bar i}}} \right|}^2}} \right)}  - C,
\end{align*}
where $C$ is a positive constant depending only on $(M,\omega)$.

Let
\[
    F (\omega_u) = \big(\sigma_k (\omega_u)\big)^{1/k}.
\]

We denote by
\[
   F^{i\bar{j}} = \frac{\partial F}{\partial w_{i\bar{j}}}, \qquad F^{i\bar{j},p\bar{q}} = \frac{\partial^2 F}{\partial w_{i\bar{j}}
    \partial w_{p\bar{q}}},
\]
where $(w_u)_{i\bar{j}} = g_{i\bar{j}} + u_{i\bar{j}}$. Then, the
positive definite matrix $(F^{i\bar{j}}(\omega_u))$ is
diagonalized at the point $x_0$.
 More precisely, we have
\begin{equation} \label{eq:1stF}
    F^{i\bar{j}}(\omega_u) = \delta_{ij} F^{i\bar{i}}(\omega_u) = \frac{1}{k} \big[\sigma_k(\lambda)\big]^{1/k-1} \sigma_{k-1}(\lambda| i)\delta_{ij}.
\end{equation}
Furthermore, at $x_0$,
\begin{equation}\label{eq:nor2ndF}
   F^{i\bar{j},p\bar{q}}(\omega_u) =
   \begin{cases}
      F^{i\bar{i},p\bar{p}}, & \textup{if $i=j$, $p=q$};\\
      F^{i\bar{p},p\bar{i}}, & \textup{if $i=q$, $p=j$, $i \ne p$};\\
      0, & \textup{otherwise},
   \end{cases}
\end{equation}
in which
\begin{align*}
  F^{i\bar{i},p\bar{p}}
  & = \frac{1}{k} \big[\sigma_k(\lambda)\big]^{1/k-1} (1 - \delta_{ip})\sigma_{k-2}(\lambda|ip) \\
  &  \quad + \frac{1}{k}(\frac{1}{k}-1) \big[\sigma_{k}(\lambda)\big]^{1/k - 2} \sigma_{k-1}(\lambda|i) \sigma_{k-1}(\lambda|p), \\
  F^{i\bar{p},p\bar{i}}
  & = - \frac{1}{k} \big[\sigma_k(\lambda)\big]^{1/k-1} \sigma_{k-2}(\lambda|ip).
\end{align*}

Here and in the follows, $\sigma_r(\lambda|i_1\ldots i_l)$, with $i_1,\ldots, i_l$ being distinct, stands for the $r$--th symmetric function with $\lambda_{i_1} = \cdots = \lambda_{i_l} = 0$.

We have, in addition at $x_0$,
\begin{equation} \label{eq:sumF}
   \sum_{i=1}^n F^{i\bar{i}}w_{i\bar{i}} = \sum_{i=1}^n F^{i\bar{i}} \lambda_i = \sigma_k^{1/k} = e^{\frac{f}{k}}.
\end{equation}

Thus by maximum principal, we have
\begin{align}\label{keyformula}
0\ge F^{i\bar j } Q_{i\bar j }=&F^{i\bar i } Q_{i\bar i }\\
\ge& (1 + 2c_0 )\sum\limits_{i = 1}^n
 {\frac{{F^{i\bar i } u_{1\bar 1 i\bar i } }}{{w_{1\bar 1 } }}}
+ \frac{{c_0 }}{2}
   \sum\limits_{i = 1}^n {\sum\limits_{p \ne 1}
    {\frac{{F^{i\bar i } |u_{1\bar p i} |^2 }}
    {{w_{1\bar 1 }^2 }}} }  \notag\\
  &- (1 + 2c_0 )\sum\limits_{i = 1}^n {\frac{{F^{i\bar i }
   |u_{1\bar 1 i} |^2 }}{{w_{1\bar 1 } ^2 }}}
    + \psi '\sum\limits_{i = 1}^n {F^{i\bar i } u_{i\bar i } }
    + \psi ''\sum\limits_{i = 1}^n {F^{i\bar i } |u_i } |^2  \notag \\
  &+ \varphi''\sum\limits_{i = 1}^n {F^{i\bar i }
    |\nabla u} |_i^2 |\nabla u|_{\bar i }^2
  +\varphi '\sum\limits_{i,p = 1}^n {F^{i\bar i }
 \left( {|u_{p\bar i } |^2  + |u_{pi} |^2 } \right)}\notag\\
 &+\varphi'\sum\limits_{i,p = 1}^n {F^{i\bar i } }
 (u_{p\bar i i} u_{\bar p }
 + u_{\bar p i\bar i } u_{p } )
 - C_1 \sum\limits_{i = 1}^n {F^{i\bar i } }\notag\\
 &:=I_1+I_2+I_3+I_4+I_5+I_6+I_7+I_8 \notag
 \end{align}
The equation can be written as
\[F(\omega_u)= e^{\frac{f}{k}}:=h\]
Differentiate the above equation , we obtain
\begin{align*}
\sum\limits_{i,j = 1}^nF^{i\bar j } u_{i\bar j l} & = \nabla _l F = h_l  ,\\
 \sum\limits_{i,j = 1}^nF^{i\bar j } u_{i\bar j l\bar m }
 &+ \sum\limits_{i,j,p,q = 1}^nF^{i\bar j ,p\bar q } u_{i\bar j l}
  u_{p\bar q \bar m}= h_{l\bar m } . \\
\end{align*}
and
\begin{align*}
\sum\limits_{i = 1}^n {F^{i\bar i } u_{i\bar i 1\bar 1 } }
&= h_{1\bar 1 }  - \sum\limits_{i,j,p,q = 1}^n
{F^{i\bar j ,p\bar q } u_{i\bar j 1}
u_{p\bar q \bar 1 }}.\\
\end{align*}
By commuting the covariant derivatives  formula $\eqref{four}$, we have
\begin{align}\label{above}
 \sum\limits_{i = 1}^nF^{i\bar i } u_{1\bar 1 i\bar i }
=& \sum\limits_{i = 1}^nF^{i\bar i } u_{i\bar i 1\bar 1 }
  + \sum\limits_{i = 1}^nF^{i\bar i } \left( {u_{1\bar 1 }
  - \sum\limits_{i = 1}^nu_{i\bar i } } \right)R_{i\bar i 1\bar 1 }  \\
  &+ \sum\limits_{i = 1}^nF^{i\bar i } \left( \sum\limits_{p = 1}^n{T_{1i}^p
u_{p\bar 1 \bar i }+ \sum\limits_{q = 1}^n\bar {T_{1i}^q }
u_{1\bar q i}- \sum\limits_{p = 1}^n|T_{1i}^p |^2 u_{p\bar p } } \right). \notag
 \end{align}
Inserting $\eqref{above}$ into the term  $I_1$, we have
\begin{align}\label{keyformula1}
 I_1=& (1 + 2c_0 )\sum\limits_{i = 1}^n {\frac{{F^{i\bar i }u_{1\bar 1 i\bar i } }}{{w_{1\bar 1 } }}}  \\
=& (1 + 2c_0 )\sum\limits_{i = 1}^n {\frac{{F^{i\bar i }
  u_{i\bar i 1\bar 1 } }}{{w_{1\bar 1 } }}}
  + (1 + 2c_0 )\sum\limits_{i = 1}^n {\frac{{F^{i\bar i }
  \left( {u_{1\bar 1 }  - u_{i\bar i } } \right)
  R_{i\bar i 1\bar 1 } }}{{w_{1\bar 1 } }}} \notag \\
  &+ 2(1 + 2c_0 )\sum\limits_{i,p = 1}^n {F^{i\bar i } }
  {\mathop{\rm Re}\nolimits} \left( {\frac{{T_{1i}^p
   u_{p\bar 1 \bar i } }}{{w_{1\bar 1 } }}} \right)
  - (1 + 2c_0 )\sum\limits_{i,p = 1}^n
   {F^{i\bar i } } \frac{{|T_{1i}^p |^2 u_{p\bar p } }}
   {{w_{1\bar 1 } }} \notag\\
=& (1 + 2c_0 )\frac{{h_{1\bar 1 } }}{{w_{1\bar 1 } }}
  - (1 + 2c_0 )\sum\limits_{i,j,p,q = 1}^n {\frac{{F^{i\bar j ,p\bar q }
   u_{i\bar j 1} u_{p\bar q \bar1} }}{{w_{1\bar 1 } }}} \notag \\
  &+ (1 + 2c_0 )\sum\limits_{i = 1}^n {\frac{{F^{i\bar i }
  \left( {u_{1\bar 1 }  - u_{i\bar i } }
  \right)R_{i\bar i 1\bar 1 } }}{{w_{1\bar 1 } }}}
   + 2(1 + 2c_0 )\sum\limits_i^n {F^{i\bar i } } {\mathop{\rm Re}\nolimits}
   \left( {\frac{{T_{1i}^1 u_{1\bar 1 \bar i } }}
   {{w_{1\bar 1 } }}}\right) \notag\\
  &+ 2(1 + 2c_0 )\sum\limits_{i = 1}^n {F^{i\bar i } }
  {\mathop{\rm Re}\nolimits} \left( {\sum\limits_{p \ne 1}
   {\frac{{T_{1i}^p u_{p\bar 1 \bar i } }}{{w_{1\bar 1 } }}} }
    \right) - (1 + 2c_0 )\sum\limits_{i,p = 1}^n {F^{i\bar i } }
    \frac{{|T_{1i}^p |^2 u_{p\bar p } }}{{w_{1\bar 1 } }} \notag\\
:=& I_{11}+I_{12}+I_{13}+I_{14}+I_{15}+I_{16} .\notag
\end{align}
Next we estimate each term of $(1)$ as follows,firstly we have
   \[I_{11}+I_{13}+I_{16}  \ge-C_1  - 3\left( {nC_2  + C_3 } \right)\sum\limits_{i = 1}^n {F^{i\bar i } } ,\]
where we have supposed that $\sup\limits_M |T|_g^2  \le C_2,\sup\limits_M |R| \le C_3 $.

 Next we claim $I_{15}+I_{2}\geq-18n^2C_2\sum\limits_{i = 1}^n{F^{i\bar i}} .$
In fact,
\begin{align*}
 I_{15}+I_{16} &= \frac{{c_0 }}{2}\sum\limits_{i = 1}^n {\sum\limits_{p \ne 1} {\frac{{F^{i\bar i } |u_{1\bar p i} |^2 }}{{w_{1\bar 1 }^2 }} + } } 2(1 + 2c_0 )\sum\limits_{i = 1}^n {F^{i\bar i } } {\mathop{\rm Re}\nolimits} \left( {\sum\limits_{p \ne 1} {\frac{{T_{1i}^p u_{p\bar 1  \bar i } }}{{w_{1\bar 1 } }}} } \right) \\
  &= \frac{{c_0 }}{2}\sum\limits_{i = 1}^n {F^{i\bar i } \sum\limits_{p \ne 1} {\left|\frac{{u_{1\bar p i} }}{{w_{1\bar 1 } }} + \frac{{2(1 + 2c_0 )}}{{c_0 }}T_{1i}^p \right|} } ^2  - \frac{{2(1 + 2c_0 )^2 }}{{c_0 }}\sum\limits_{i = 1}^n {\sum\limits_{p \ne 1} {F^{i\bar i } |T_{1i}^p |^2 } }  \\
  &\ge  - \frac{{2(1 + 2c_0 )^2 }}{{c_0 }}\sum\limits_{i = 1}^n {\sum\limits_{p \ne 1} {F^{i\bar i } |T_{1i}^p |^2 } }  \\
  &\ge  -18n^2 C_2 \sum\limits_{i = 1}^n {F^{i\bar i } },
 \end{align*}
where we have used $\frac{1}{n^2}\le c_0\le1 $
Thus, we obtain,
\begin{align}\label{keyformula2}
 I_1+I_2\ge&  - (1 + 2c_0 )\sum\limits_{i,j,p,q = 1}^n {\frac{{F^{i\bar j ,p\bar q } u_{i\bar j 1} u_{p\bar q \bar1} }}{{w_{1\bar 1 } }}}  + 2(1 + 2c_0 )\sum\limits_i^n {F^{i\bar i } } {\mathop{\rm Re}\nolimits} \left( {\frac{{T_{1i}^1 u_{1\bar 1 \bar i } }}{{w_{1\bar 1 } }}} \right) \\&- \left( {21n^2C_2  +3 C_3 } \right)\sum\limits_{i = 1}^n {F^{i\bar i } }  - C_1 .\notag
\end{align}
 For terms  $I_7+I_8$, we claim
 \begin{align}\label{keyformula3}
I_7+I_8 \ge \frac{1}{2}\varphi '\sum\limits_{i = 1}^n
 {F^{i\bar i } |u_{i\bar i } |^2 _1  - \left( {C_2  + C_3 } \right)
 \sum\limits_{i = 1}^n {F^{i\bar i } } - C_1 } .
\end{align}

 Indeed, by the covariant derivatives' commutation formula $\eqref{three}$ in section 2, we have
 \begin{align*}
 u_{pi\bar i }  =u_{i\bar i p}  + T_{pi}^i u_{i\bar i }  + u_q R_{i\bar i p\bar q } ,
 u_{\bar p i\bar i }  = u_{i\bar p \bar i }  = u_{i\bar i \bar p }  - \overline {T_{ip}^i } u_{i\bar i } . \\
\end{align*}
 Then we have
 \begin{align*}
 \sum\limits_{i = 1}^n {F^{i\bar i } } u_{pi\bar i }  &= \sum\limits_{i = 1}^n {F^{i\bar i } } u_{i\bar i p}  + \sum\limits_{i = 1}^n {F^{i\bar i } } \left( {T_{pi}^i u_{i\bar i }  + u_q R_{i\bar i p\bar q } } \right) = F_p  + \sum\limits_{i = 1}^n {F^{i\bar i } } \left( {T_{pi}^i u_{i\bar i }  + u_q R_{i\bar i p\bar q } } \right) \\
\sum\limits_{i = 1}^n {F^{i\bar i } } u_{\bar p i\bar i }  &= \sum\limits_{i = 1}^n {F^{i\bar i } } u_{i\bar i \bar p }  + \sum\limits_{i = 1}^n {F^{i\bar i } } \overline {T_{ip}^i } u_{i\bar i } = F_{\bar p }  + \sum\limits_{i = 1}^n {F^{i\bar i } } \overline {T_{ip}^i } u_{i\bar i }  \\
 \end{align*}
Inserting the above formula into the term $(8)$, we obtain
 \begin{align}\label{keyformula4}
 I_8& = \varphi '\sum\limits_{i,p = 1}^n {F^{i\bar i } } (u_{p\bar i i} u_{\bar p }  + u_{\bar p i\bar i } u_p ) \\
  &= \varphi '\sum\limits_{p = 1}^n {u_{\bar p } \left[ {F_p  + \sum\limits_{i = 1}^n {F^{i\bar i } } \left( {T_{pi}^i u_{i\bar i }  + u_q R_{i\bar i p\bar q } } \right)} \right]}  + \varphi '\sum\limits_{p = 1}^n {u_p \left[ {h_{\bar p }  - \sum\limits_{i = 1}^n {F^{i\bar i } } \overline {T_{ip}^i } u_{i\bar i } } \right]}  \notag\\
  &= 2\varphi '\sum\limits_{i,p = 1}^n {F^{i\bar i } u_{i\bar i } {\mathop{\rm Re}\nolimits} \left( {u_{\bar p } T_{pi}^i } \right)}  + \varphi '\sum\limits_{p = 1}^n {\left[ {2{\mathop{\rm Re}\nolimits} \left( {u_{\bar p } h_p } \right) + \sum\limits_{i,q = 1}^n {u_{\bar p } u_q F^{i\bar i } R_{i\bar i p\bar q } } } \right]}  \notag\\
  &= I_{81}+I_{82} .\notag
 \end{align}
For the term $I_{82}$, we have
\[
 I_{82}  \ge  - C_3 \sum\limits_{i = 1}^n {F^{i\bar i } }  - C_1  .\]
 For the term $I_{81}$, we obtain
 \begin{align*}
 I_{81}  + I_{7} &= 2\varphi '\sum\limits_{i,p = 1}^n {F^{i\bar i } u_{i\bar i } {\mathop{\rm Re}\nolimits} \left( {u_{\bar p } T_{pi}^i } \right)}  + \varphi '\sum\limits_{i,p = 1}^n {F^{i\bar i } \left( {|u_{p\bar i } |^2  + |u_{pi} |^2 } \right)}  \\
  &\ge \varphi '\sum\limits_{i = 1}^n {F^{i\bar i } \left[ {|u_{i\bar i } |^2  + 2u_{i\bar i } {\mathop{\rm Re}\nolimits} \left( {\sum\limits_{p = 1}^n {u_{\bar p } T_{pi}^i } } \right)} \right]}  \\
  &= \varphi '\sum\limits_{i = 1}^n {F^{i\bar i } \left|\frac{{u_{i\bar i } }}{2} + 2\sum\limits_{p = 1}^n {u_p \overline {T_{pi}^i } } \right|^2 }  + \frac{3}{4}\varphi '\sum\limits_{i = 1}^n {F^{i\bar i } \left| {u_{i\bar i } } \right|^2 }  - 4\varphi '\sum\limits_{i = 1}^n {F^{i\bar i } } \left| {\sum\limits_{p = 1}^n {u_p \overline {T_{pi}^i } } } \right|^2  \\
  &\ge \frac{1}{2}\varphi '\sum\limits_{i = 1}^n {F^{i\bar i } \left| {u_{i\bar i } } \right|^2 }  - C_2 \sum\limits_{i = 1}^n {F^{i\bar i } }.
\end{align*}
Thus we have proved the above claim $\eqref{keyformula3}$.

Moreover, apply $\eqref{eq:sumF}$ to obtain
\begin{align*}
\psi '\sum\limits_{i = 1}^n {F^{i\bar i } u_{i\bar i } }
=\psi '\sum\limits_{i = 1}^n {F^{i\bar i }(\lambda_i-1)} =\psi 'h-\psi '\sum\limits_{i = 1}^n F^{i\bar i }
\ge -2(C_0+1)\sup_M{{e^{\frac {f}{k}}}}+\psi '\sum\limits_{i = 1}^n F^{i\bar i }
\end{align*}
Similarly,
\begin{align*}
\frac{1}{2}\varphi '&\sum\limits_{i = 1}^n {F^{i\bar i} \left| {u_{i\bar i} } \right|^2 }  = \frac{1}{2}\varphi '\sum\limits_{i = 1}^n {F^{i\bar i} \left( {\lambda _i  - 1} \right)^2  }\\
&=  \frac{1}{2}\varphi '\sum\limits_{i = 1}^n {F^{i\bar i} \lambda _i ^2 }  - \varphi '\sum\limits_{i = 1}^n {F^{i\bar i} \lambda _i }  + \frac{1}{2}\varphi '\sum\limits_{i = 1}^n {F^{i\bar i} }  \\
  &= \frac{1}{2}\varphi '\sum\limits_{i = 1}^n {F^{i\bar i} \lambda _i ^2 }  - \varphi 'h + \frac{1}{2}\varphi '\sum\limits_{i = 1}^n {F^{i\bar i} }\\
 &\ge \frac{1}{2}\varphi '\sum\limits_{i = 1}^n {F^{i\bar i} \lambda _i ^2 }  - \frac{1}{2}\sup_M{{e^{\frac {f}{k}}}} + \frac{1}{2}\varphi '\sum\limits_{i = 1}^n {F^{i\bar i} }\\
 \end{align*}
Inserting these terms into \eqref{keyformula}, we obtain
\begin{align}\label{FQ}
 0 \ge F^{i\bar i } Q_{i\bar i }  \ge&
 - (1 + 2c_0 )\sum\limits_{i,j,p,q = 1}^n {\frac{{F^{i\bar j ,p\bar q } u_{i\bar j  1} u_{p\bar q \bar 1} }}{{w_{1\bar 1 } }}}  + 2(1 + 2c_0 )\sum\limits_{i = 1}^n {F^{i\bar i } } {\mathop{\rm Re}\nolimits} \left( {\frac{{T_{1i}^1 u_{1\bar 1 \bar i } }}{{w_{1\bar 1 } }}} \right)
  - (1 + 2c_0 )\sum\limits_{i = 1}^n {\frac{{F^{i\bar i }
  |u_{1\bar 1 i} |^2 }}{{w_{1\bar 1 } ^2 }}}\\
  &+ \varphi ''\sum\limits_{i = 1}^n {F^{i\bar i } |\nabla u} |_i^2 |\nabla u|_{\bar i }^2
  + \psi ''\sum\limits_{i = 1}^n {F^{i\bar i }|u_i|^2}+ \frac{1}{2}\varphi '\sum\limits_{i = 1}^n
   {F^{i\bar i } \lambda_i^2 }\\
    &+\left( {-\psi '+\frac{1}{2}\varphi ' -22n^2 C_2  - 4C_3 } \right)\sum\limits_{i = 1}^n
   {F^{i\bar i } }  - C_1  \notag\\
  =&A_1+A_2+A_3\notag\\&+A_4+A_5+A_6+\left( {-\psi '+\frac{1}{2}\varphi ' -22n^2 C_2  - 4C_3 } \right)\sum\limits_{i = 1}^n
   {F^{i\bar i } }  - C_1  \notag,
 \end{align}
 where $C_1$ is a positive constant depending only on $C_0$, $\sup{e^{\frac{f}{k}}}$, and $\sup\left|\nabla \left(e^{\frac{f}{k}}\right)\right|^2$ , and $\sup\left|\partial\bar\partial \left(e^{\frac{f}{k}}\right)\right|$.

Let $\varepsilon  = \frac{\delta}{4}\le \frac{1}{16}$ and $\delta=\frac{1}{2A + 1}$ , where $A=2L(C_0+1)$ and $C_0=31n^2C_2+4C_3$. We divide two cases to drive the estimate, which is similar as \cite{HMW}.

 \textbf{Case1:} $\lambda_n < -\varepsilon\lambda_1.$

 By the first derivative's condition \eqref{firstordercondition} , we have

\begin{align*}
- (1 + 2c_0 )^2 \left| {\frac{{u_{1\bar 1 i} }}{{w_{1\bar 1 } }}} \right|^2  &=  - \left| {\varphi '|\nabla u|_i^2  + \psi 'u_i } \right|^2 \ge  - 2\left( {\varphi '} \right)^2 |\nabla u|_i^2 |\nabla u|_{\bar i }^2  - 2\left( {\psi '} \right)^2 |u_i |^2 \notag \\
&=  - \varphi ''|\nabla u|_i^2 |\nabla u|_{\bar i }^2  - 2\left( {\psi '} \right)^2 |u_i |^2,1\leq i\leq n
\end{align*}

  \begin{align*}
 A_2  &= 2(1 + 2c_0 )\sum\limits_{i \ne 1} {F^{i\bar i } {\mathop{\rm Re}\nolimits} \left( {\frac{{T_{1i}^1 u_{1\bar 1 \bar i } }}{{w_{1\bar 1 } }}} \right)}  \\
  &\ge  - c_0 \sum\limits_{i \ne 1} {F^{i\bar i } \left| {\frac{{u_{1\bar 1 \bar i } }}{{w_{1\bar 1 } }}} \right|^2  - \frac{{(1 + 2c_0 )^2 }}{{c_0 }}\sum\limits_{i \ne 1} {F^{i\bar i } \left| {T_{1i}^1 } \right|^2 } }\\
  &\ge  - c_0 \sum\limits_{i \ne 1} {F^{i\bar i } \left| {\frac{{u_{1\bar 1 \bar i } }}{{w_{1\bar 1 } }}} \right|^2  - 9n^2C_2\sum\limits_{i \ne 1} {F^{i\bar i } \left| {T_{1i}^1 } \right|^2 } }
  \end{align*}
Thus
  \begin{align*}
 A_2  + A_3  &\ge  - (1 + 3c_0 )\sum\limits_{i = 1}^n {\frac{{F^{i\bar i } |u_{1\bar 1 i} |^2 }}{{w_{1\bar 1 } ^2 }}}  - 9n^2C_2\sum\limits_{i \ne 1} {F^{i\bar i } \left| {T_{1i}^1 } \right|^2 }  \\
  &\ge  - (1 + 2c_0 )^2 \sum\limits_{i = 1}^n {\frac{{F^{i\bar i } |u_{1\bar 1 i} |^2 }}{{w_{1\bar 1 } ^2 }}}  - 9n^2C_2 \sum\limits_{i = 1}^n {F^{i\bar i } }  \\
  &=  - A_4  - 2\left( {\psi '} \right)^2 \sum\limits_{i = 1}^n {F^{i\bar i } |u_i |^2 }  - 9n^2C_2 \sum\limits_{i = 1}^n {F^{i\bar i } }.
  \end{align*}

 We therefore obtain
 \begin{align}\label{A_234}
A_2  + A_3  + A_4  \ge  - 2\left( {\psi '} \right)^2 \sum\limits_{i = 1}^n
{F^{i\bar i } |u_i |^2 }  - 9n^2C_2 \sum\limits_{i = 1}^n {F^{i\bar i } }.
 \end{align}


Using the following inequlity
 \begin{align*}
\sum\limits_{i = 1}^n {F^{i\bar i} \lambda _i ^2 }  \ge F^{n\bar n} \lambda _n ^2  > \varepsilon^2 F^{n\bar n} \lambda _1 ^2 \ge \frac{\varepsilon^2}{n}\sum\limits_{i = 1}^n {F^{i\bar i} } \lambda _1 ^2.
\end{align*}
Therefore, we have
\begin{align}\label{A_6}
A_6=\frac{1}{2}\varphi '\sum\limits_{i = 1}^n {F^{i\bar i} \lambda _i ^2 }
\ge \frac{\varepsilon^2}{2n}\varphi '\sum\limits_{i = 1}^n {F^{i\bar i} } \lambda _1 ^2
 \end{align}

  Combining \eqref{FQ} and \eqref{A_234} \eqref{A_6}, we obtain
 \begin{align*}
 0 \ge& \sum\limits_{i = 1}^n {F^{i\bar i} Q_{i\bar i} }  \ge  \frac{\varepsilon^2}{2n}\varphi '\sum\limits_{i = 1}^n {F^{i\bar i} } \lambda _1 ^2   -2\left( {\psi '} \right)^2\sum\limits_{i = 1}^n {F^{i\bar i}|u_i |^2 }  \\
  &+ \left( {-\psi '+\frac{1}{2}\varphi '  -31n^2C_2  -4C_3   } \right)\sum\limits_{i = 1}^n {F^{i\bar i} }  - C_1  \\
&\geq \left(\frac{\varepsilon^2}{8nK}\lambda_1^2-8K(C_0  + 1)^2\right)
\sum\limits_{i = 1}^n{F^{i\bar i }}- C_1\\
&\geq \frac{\varepsilon^2}{8nK}\lambda_1^2-8K(C_0  + 1)^2- C_1,
 \end{align*}
 where we have used the fact that $ \sum\limits_{i = 1}^n{F^{i\bar i }}\geq 1$, which follows from Newton-Maclaurin's inequality and the fact that .

Hence,we obtain the estimates
\begin{center}
$\lambda_1\leq 8\sqrt{2}(2A+1)\sqrt{nK(8K(C_0  + 1)^2+ C_1)}\le CK .$
\end{center}

\textbf{Case2:} $\lambda_n > -\varepsilon\lambda_1$.

Let
\[
I = \left\{ {i \in \left\{ {1, \cdots ,n} \right\}\left| {\sigma _{k - 1} \left( {\lambda \left| i \right.} \right) \ge \varepsilon ^{ - 1} \sigma _{k - 1} \left( {\lambda \left| 1 \right.} \right)} \right.} \right\}
\]
Obviously, $1 \notin I $ and $i\in I$ if and only if
\[
F^{i\bar i}  > \varepsilon ^{ - 1} F^{1\bar 1}
\]
We first treat those indices which are not in $ I $: by the first derivative's condition \eqref{firstordercondition}, we have
\begin{align*}
 -(1 + 2c_0 )& \sum\limits_{i \notin I}{\frac{{F^{i\bar i } |u_{1\bar 1 i} |^2 }}{{w_{1\bar 1 } ^2 }}}+2(1+2c_0)\sum\limits_{i \notin I}{F^{i\bar i} Re\frac{ T_{1i}^1 u_{1\bar 1 \bar i}}{w_{1\bar 1} }}\\
 &\ge  -(1 + 2c_0 )^2 \sum\limits_{i \notin I}{\frac{{F^{i\bar i } |u_{1\bar 1 i} |^2 }}{{w_{1\bar 1 } ^2 }}}-\frac{(1 + 2c_0 )^2}{{c_0}} \sum\limits_{i \notin I}{{F^{i\bar i } |T_{1i}^1 |^2 }}\\
&=-\varphi''\sum\limits_{i\notin I}{F^{i\bar i}{|\nabla u|^2}_i{|\nabla u|^2}_{\bar i}}-2(\psi')^2\sum\limits_{i\notin I}{F^{i\bar i}|u_i|^2} -9n^2C_2 \sum\limits_{i \notin I}{{F^{i\bar i } |T_{1i}^1 |^2 }}\\
&\ge -\varphi''\sum\limits_{i\notin I}{F^{i\bar i}{|\nabla u|^2}_i{|\nabla u|^2}_{\bar i}}-2\varepsilon ^{ - 1}K(\psi')^2F^{1\bar 1} -9n^2C_2 \sum\limits_{i =1}^n{F^{i\bar i }}
\end{align*}
Substitute the above inequality into \eqref{FQ}
\begin{align}\label{Case2}
0 \ge F^{i\bar i } Q_{i\bar i }  \ge&
 - (1 + 2c_0 )\sum\limits_{i,j,p,q = 1}^n {\frac{{F^{i\bar j ,p\bar q } u_{i\bar j  1} u_{p\bar q \bar 1} }}{{w_{1\bar 1 } }}}  + 2(1 + 2c_0 )\sum\limits_{i \in I } {F^{i\bar i } } {\mathop{\rm Re}\nolimits} \left( {\frac{{T_{1i}^1 u_{1\bar 1 \bar i } }}{{w_{1\bar 1 } }}} \right) \\
  &- (1 + 2c_0 )\sum\limits_{i \in I} {\frac{{F^{i\bar i }
  |u_{1\bar 1 i} |^2 }}{{w_{1\bar 1 } ^2 }}}+ \varphi ''\sum\limits_{i \in I} {F^{i\bar i } |\nabla u} |_i^2 |\nabla u|_{\bar i }^2
  + \psi ''\sum\limits_{i = 1}^n {F^{i\bar i } |u_i } |^2\notag
  \\
      &+ \frac{1}{2}\varphi '\sum\limits_{i = 1}^n
   {F^{i\bar i } \lambda_i^2 }
   -2\varepsilon ^{ - 1}K(\psi')^2F^{1\bar 1}+ \left( {-\psi '+\frac{1}{2}\varphi ' -31n^2C_2  - 4C_3 } \right)\sum\limits_{i = 1}^n
   {F^{i\bar i } }  - C_1  \notag\\
   &=B_1+B_2+B_3+B_4+B_5\notag\\
   &+B_6+B_7+B_8\notag
 \end{align}
 Firstly, we have
 \label{B67}\[
 B_6+B_7=\frac{1}{2}\varphi '\sum\limits_{i = 1}^n
   {F^{i\bar i } \lambda_i^2 }
   -2\varepsilon ^{ - 1}K(\psi')^2F^{1\bar 1}\ge\frac{1}{4}\varphi '\sum\limits_{i = 1}^n
   {F^{i\bar i } \lambda_i^2 },
 \]
 where we have assumed $\frac{1}{4}\varphi '
   {F^{1\bar 1 } \lambda_1^2 }\ge
   2\varepsilon ^{ - 1}K(\psi')^2F^{1\bar 1} $ otherwise we have $\frac{1}{4}\varphi '
   {F^{1\bar 1 } \lambda_1^2 }\le
   2\varepsilon ^{ - 1}K(\psi')^2F^{1\bar 1} $ i.e. $\lambda_1\leq CK$ and the estimate is done.

We next use the first term $B_1$ to cancel  the other terms containing the third derivatives of $u$ .
By the same proof as in \cite{HMW} P559, we have
\[
\lambda_1\sigma _{k - 2} \left( {\lambda \left| {1i} \right.} \right) \ge \left( {1 - 2\varepsilon } \right)\sigma _{k - 1} \left( {\lambda \left| i \right.} \right), \text{for}\  i \in I.
\]
Thus \begin{align*}
  - \lambda _1 F^{i\bar 1,1\bar i}  = \frac{{F^{1 - k} }}{k}\lambda _1 \sigma _{k - 2} \left( {\lambda \left| {1i} \right.} \right) \ge \frac{{F^{1 - k} }}{k}\left( {1 - 2\varepsilon } \right)\sigma _{k - 1} \left( {\lambda \left| i \right.} \right) = \left( {1 - 2\varepsilon } \right)F^{i\bar i}  \\
 \end{align*}
Since
\begin{align*}
 u_{i\bar 11}  = u_{1\bar 1i}  - T_{1i}^1 \left( {\lambda _1  - 1} \right) \\
 \end{align*}
Therefore
 \begin{align*}
  B_1=- \frac{{{\rm{1 + 2c}}_{\rm{0}} }}{{\lambda _1 }}\sum\limits_{i,j,p,q = 1}^n {F^{i\bar j,p\bar q} } u_{i\bar j1} u_{p\bar q\bar 1}  \ge  - \frac{{{\rm{1 + 2c}}_{\rm{0}} }}{{\lambda _1 ^2 }}\sum\limits_{i\in I} {\lambda _1 F^{i\bar 1,1\bar i} } u_{i\bar 11} u_{1\bar i\bar 1}  \\
   \ge \frac{{{\rm{1 + 2c}}_{\rm{0}} }}{{\lambda _1 ^2 }}\left( {1 - 2\varepsilon } \right)\sum\limits_{i \in I} {F^{i\bar i} \left| {u_{1\bar 1i}  - T_{1i}^1 \left( {\lambda _1  - 1} \right)} \right|} ^2  \\
 \end{align*}

 \begin{align*}
B_2  = \frac{{{\rm{2}}\left( {{\rm{1 + 2c}}_{\rm{0}} } \right)}}{{\lambda _1 }}\sum\limits_{i \in I} {F^{i\bar i} {\mathop{\rm Re}\nolimits} \left( {T_{1i}^1 u_{1\bar 1\bar i} } \right)}
\end{align*}

From  the first derivative's condition \eqref{firstordercondition}, we have
 \begin{align*}
B_4= \varphi ''\sum\limits_{i \in I} {F^{i\bar i }
 |\nabla u} |_i^2 |\nabla u|_{\bar i }^2  =& 2\sum\limits_{i \in I}
 {F^{i\bar i} \left|{(1 + 2c_0 )\frac{{u_{1\bar1\bar i}}}
 {{w_{1\bar 1 } }} + \psi 'u_i } \right|} ^2 \\
 \ge& 2(1 + 2c_0 )^2 \delta \sum\limits_{i \in I} {F^{i\bar i }
  \frac{{\left| {u_{1\bar 1 \bar i } } \right|}}
  {{w_{_{1\bar 1 } }^2 }}} ^2  - \frac{{2\delta }}
  {{1 - \delta }}\left( {\psi '} \right)^2 \sum\limits_{i \in I}
  {F^{i\bar i}|u_i } |^2  \\
  \ge& 2(1 + 2c_0 )^2 \delta \sum\limits_{i \in I} {F^{i\bar i }
  \frac{{\left| {u_{1\bar 1 \bar i } } \right|}}
  {{w_{_{1\bar 1 } }^2 }}} ^2  -   B_5,
 \end{align*}
 where we have used $\frac{{2\delta }}{{1 - \delta }}\left( {\psi '} \right)^2=\psi''$ by our choosing $\delta=\frac{1}{2A+1}$.

 Thus we have
\begin{align*}
B_3  + B_4+ B_5     \ge  - \left( {{\rm{1 + 2c}}_{\rm{0}} } \right)\frac{{\left[ {1 - 2\left( {{\rm{1 + 2c}}_{\rm{0}} } \right)\delta } \right]}}{{\lambda _1 ^2 }}\sum\limits_{i \in I } {F^{i\bar i} } \left| {u_{1\bar 1\bar i} } \right|^2.
\end{align*}
Therefore,
\begin{align*}
 & B_1  + B_2  + B_3  + B_4 + B_5   \\
   \ge& \frac{{{\rm{1 + 2c}}_{\rm{0}} }}{{\lambda _1 ^2 }}\left( {1 - 2\varepsilon } \right)\sum\limits_{i \in I} {F^{i\bar i} \left| {u_{1\bar 1i}  - T_{1i}^1 \left( {\lambda _1  - 1} \right)} \right|^2 }
  - \left( {{\rm{1 + 2c}}_{\rm{0}} } \right)\frac{{\left[ {1 - 2\left( {{\rm{1 + 2c}}_{\rm{0}} } \right)\delta } \right]}}{{\lambda _1 ^2 }}\sum\limits_{ i \in I} {F^{i\bar i} } \left| {u_{1\bar 1\bar i} } \right|^2 \\
   &+ \frac{{{\rm{2}}\left( {{\rm{1 + 2c}}_{\rm{0}} } \right)}}{{\lambda _1 ^2 }}\sum\limits_{i \in I} {F^{i\bar i} {\mathop{\rm Re}\nolimits} \left( {\lambda _1 T_{1i}^1 u_{1\bar 1\bar i} } \right)}  \\
  =& \frac{{{\rm{1 + 2c}}_{\rm{0}} }}{{\lambda _1 ^2 }}\sum\limits_{i \in I} {F^{i\bar i} \left\{ {\left( {1 - 2\varepsilon } \right)\left| {u_{1\bar 1i}  - T_{1i}^1 \left( {\lambda _1  - 1} \right)} \right|^2  - \left( {1 - 2\left( {{\rm{1 + 2c}}_{\rm{0}} } \right)\delta } \right)\left| {u_{1\bar 1i} } \right|^2  + 2{\mathop{\rm Re}\nolimits} \left( {\lambda _1 T_{1i}^1 u_{1\bar 1\bar i} } \right)} \right\}}  \\
  =& \frac{{{\rm{1 + 2c}}_{\rm{0}} }}{{\lambda _1 ^2 }}\sum\limits_{i \in I} {F^{i\bar i} \left\{ {\left( {2\left( {{\rm{1 + 2c}}_{\rm{0}} } \right)\delta  - 2\varepsilon } \right)\left| {u_{1\bar 1i} } \right|^2  + 2\left[ {2\varepsilon \left( {\lambda _1  - 1} \right) + 1} \right]{\mathop{\rm Re}\nolimits} \left( {T_{1i}^1 u_{1\bar 1\bar i} } \right) + \left( {1 - 2\varepsilon } \right)\left( {\lambda _1  - 1} \right)^2 \left| {T_{1i}^1 } \right|^2 } \right\}}  \\
  &\ge 0,
 \end{align*}
where the last inequality holds
 if we choose $\varepsilon  = \frac{\delta}{4}\le \frac{1}{16}$ . In fact,
\begin{align*}
\Delta= B^2  - 4AC =& 4\left[ {2\varepsilon \left( {\lambda _1  - 1} \right) + 1} \right]^2  - 4\left( {1 - 2\varepsilon } \right)\left( {\lambda _1  - 1} \right)^2 \left( {2\left( {{\rm{1 + 2c}}_{\rm{0}} } \right)\delta  - 2\varepsilon } \right) \\
  &\le 36\varepsilon ^2 \left( {\lambda _1  - 1} \right)^2  - 4\left( {1 - 2\varepsilon } \right)\left( {\lambda _1  - 1} \right)^2\left( {2\left( {{\rm{1 + 2c}}_{\rm{0}} } \right)\delta  - 2\varepsilon } \right) \\
  &\le 4\left( {\lambda _1  - 1} \right)^2 \left( {9\varepsilon ^2  - 2\left( {1 - 2\varepsilon } \right)\left( {\left( {{\rm{1 + 2c}}_{\rm{0}} } \right)\delta   } \right)} + 2\varepsilon \left( {1 - 2\varepsilon } \right)\right) \\
  &\le 4\left( {\lambda _1  - 1} \right)^2 \left( 5\varepsilon ^2 + 2\varepsilon   -{\delta   } \right) \\
  &\le4\left( {\lambda _1  - 1} \right)^2 \left(  4\varepsilon   -{\delta   } \right) \\
  &=0.
 \end{align*}
Thus we finally obtain
\begin{align*}
 0 \ge& \frac{1}{4}\varphi '\sum\limits_{i = 1}^n
   {F^{i\bar i } \left| {u_{i\bar i } } \right|^2 }
    +\left( {-\psi '+\frac{1}{2}\varphi ' - C_2  - C_3 } \right)\sum\limits_{i = 1}^n
   {F^{i\bar i } }  - C_1\\
  =&  \left( {-\psi '+\frac{1}{2}\varphi ' - C_2  - C_3 } \right)\sum\limits_{i = 1}^n {F^{i\bar i} }  + \frac{1}{2}\varphi '\sum\limits_{i = 1}^n {F^{i\bar i} \left| {u_{i\bar i} } \right|^2 }  - C_1  \\
  \ge& \sum\limits_{i = 1}^n {F^{i\bar i} }  + \frac{1}{16K}\sum\limits_{i = 1}^n {F^{i\bar i} {\lambda _i}^2 }  - C_1  \\
 \end{align*}
, where we have used $-\psi'\ge C_0+1$ by choosing  $C_0= 31n^2C_2  + 4C_3$ .

In particular, we have $\sum\limits_{i = 1}^n {F^{i\bar i} }  \le C$.
By lemma$2.2$ in \cite{HMW} we have $ F^{1\bar 1}   \ge \frac {c(n,k)}{C_1^{k-1}} $, where $c(n,k)$ is a positive constant depending only on $n$ and $k$.

Therefore, we get the desired estimate:
\begin{align*}
\lambda _1  \le  \frac{4C_1^{\frac{k}{2}}}{c(n,k)^{\frac{1}{2}}}\sqrt K,
\end{align*}
where $C_1$ is given in \eqref{FQ}.
\end{proof}

\end{document}